\documentclass{amsart}

\usepackage{graphicx}
\usepackage{amsrefs}
\usepackage{bm}
\usepackage{amsmath}
\usepackage{amssymb}
\usepackage{xcolor}
\usepackage[normalem]{ulem}

 \newtheorem{theorem}{Theorem}[section]
\newtheorem{lemma}[theorem]{Lemma}

\theoremstyle{definition}
\newtheorem{definition}[theorem]{Definition}

\theoremstyle{remark}

\numberwithin{equation}{section}

\newcommand{\IGNORE}[1]{}
\newcommand{\ignore}[1]{}

\newcommand{\mbb}[1]{\mathbb{#1}}
\newcommand{\mb}[1]{\mathbf{#1}}
\newcommand{\mc}[1]{\mathcal{#1}}

\newcommand{\der}[2]{\frac{\partial #1}{\partial #2}}

\newcommand{\pa}{\partial}

\begin{document}

\title[Convergence of a mass-lumped finite element method]{Convergence of a mass-lumped finite element method for the Landau-Lifshitz equation}

\author{Eugenia Kim}
\address{Department
    of Mathematics, University of California, Berkeley,
    CA 94720 and  MS-B284, Los Alamos National Laboratory, Los Alamos, NM 87544}
\email{kim107@math.berkeley.edu}
\thanks{The first author was supported in part 
    by the U.S. Department of Energy, Office of Science, Office of
    Workforce Development for Teachers and Scientists, Office of
    Science Graduate Student Research (SCGSR) program. The SCGSR
    program is administered by the Oak Ridge Institute for Science and
    Education for the DOE under contract number DE-AC05-06OR23100.}

\author{Jon Wilkening}
\address{Department of Mathematics and Lawrence
    Berkeley National Laboratory, University of California, Berkeley,
    CA 94720}
\email{wilkening@berkeley.edu}
\thanks{The second author was supported
    in part by the US Department of Energy, Office of Science, Applied
    Scientific Computing Research, under award number
    DE-AC02-05CH11231.}





\begin{abstract}
The dynamics of the magnetic distribution in a ferromagnetic
  material is governed by the Landau-Lifshitz equation, which
  is a nonlinear geometric dispersive equation with a nonconvex
  constraint that requires the magnetization to remain of unit
    length throughout the domain. In this article, we present a mass-lumped
  finite element method for the Landau-Lifshitz equation. This method preserves the
  nonconvex constraint at each node of the finite element mesh, and is
  energy nonincreasing. We show that the numerical solution of our
  method for the Landau-Lifshitz equation converges to a weak solution
  of the Landau-Lifshitz-Gilbert equation using a simple proof
  technique that cancels out the product of weakly convergent
  sequences. Numerical tests for both explicit and implicit
  versions of the method on a unit square with periodic
  boundary conditions are provided for structured and unstructured
  meshes.
\end{abstract}

\maketitle

Micromagnetics is the study of the behavior of ferromagnetic
  materials at sub-micron length scales, including magnetization
  reversal and hysteresis effects \cite{fidler2000micromagnetic}.
  The dynamics of the magnetic
distribution of a ferromagnetic material occupying a region
$\Omega \subset \mbb{R}^2$ or $ \mbb{R}^3$ is governed by the Landau-Lifshitz (LL) equation \cite{fidler2000micromagnetic,kruzik2006recent,garcia2007numerical}.
 The magnetization $m(x,t) : \Omega \times [
  0, T] \to \mbb{S}^2 \subset \mbb{R}^3$ satisfies
\begin{equation}\label{eq:LL}
 \begin{cases}
   \partial_t m = - m \times h - \alpha m \times ( m \times h)
   & \text { in } \Omega\\
 \der{m}{\nu} =0 & \text { on } \partial \Omega \\
 m(x,0) = m_0(x)
 \end{cases}
\end{equation}
where $\alpha$ is a dimensionless damping parameter and
$h$ is an effective field given by
\begin{equation}\label{eq:h:def}
h(m) := - \frac{\delta   \mc{E}}{\delta m} (m) = \eta \Delta m - Q(m_2 e_2 + m_3 e_3) +h_s(m)+h_e.
\end{equation}
Here $\eta$ is the exchange constant, $Q$ is an anisotropy constant,
$h_s$ is the stray field, $h_e$ is an external field, and
 $\frac{\delta   \mc{E}}{\delta m} $ is the functional derivative of the Landau-Lifshitz energy, defined by
\begin{equation}\label{eq:LLenergydef}
  \mc{E}(m) = \frac{\eta}{2} \int_\Omega |\nabla m|^2   +
  \frac{Q}{2} \int_\Omega m_2^2+ m^2_3 -
  \frac{1}{2}\int_\Omega h_s(m) \cdot m -\int_\Omega h_e \cdot m .
\end{equation}
The first term is the exchange energy, which tries to align the
magnetization locally; the second term is the anisotropy energy, which
tries to orient the magnetization in certain easy direction taken to be $e_1$; the third
term is the stray field energy, which is induced by the magnetization
distribution inside the material; and the last term is the external
field energy, which tries to align the magnetization with an external
field.  We denote the lower order terms in
(\ref{eq:h:def}) by
\begin{equation}\label{eq:h:low}
\bar h(m) := - Q(m_2 e_2 + m_3 e_3) +h_s(m)+h_e.
\end{equation}
  When considering mathematical properties
such as existence and regularity of
the solution, these terms can be considered lower order compared to
the exchange term \cite{alouges2006convergence}. They also
have fewer derivatives than the exchange term, and can be treated as lower order 
when developing numerical methods. 

The stray
field $h_s$ depends on $m$ via
$h_s = - \nabla \phi$, where the potential $\phi$ satisfies
\begin{equation}
\begin{aligned}
&\Delta \phi = \begin{cases}
 \nabla \cdot m &\text{ in } \Omega \\
0 &\text{ on } \partial \Omega \\
\end{cases}\\
&\left[ \phi\right]_{\partial \Omega} =0, \qquad
  \left[ \frac{\partial \phi}{\partial \nu}\right]_{\partial \Omega}
  =- m \cdot \nu.
\end{aligned}
\end{equation}
Here $[v]_{\partial \Omega} (x) = v(x^+) - v(x^-)$ is the jump
  in $v$ across the boundary $\pa\Omega$ from inside ($-$) to
  outside ($+$); see \cite{garcia2007numerical}.

There are several equivalent forms of the Landau-Lifshitz (LL) equation.
 The following is the Landau-Lifshitz-Gilbert (LLG) equation :
\begin{equation} \label{eq:LLG}
\begin{aligned}
        \partial_t m- \alpha m \times       \partial_t m  = -(1+\alpha^2) (m \times h).
\end{aligned}
\end{equation}
Also, the equation
\begin{equation} \label{eq:LLGdifferent}
\begin{aligned}
\alpha       \partial_t m +  m \times       \partial_t m  = (1+\alpha^2) (h - (h \cdot m) m)
\end{aligned}
\end{equation}
is equivalent to LL and LLG; see \cite{alouges2006convergence}.
If only the gyromagnetic term is present in equation (\ref{eq:LL}),
i.e.~if $   \partial_t m= - m \times \Delta m$, it is called a
Schr\"{o}dinger map into $\mathbb{S}^2$ \cite{gustafson2010asymptotic}.
This is a geometric
generalization of the linear Schr\"{o}dinger equation. If only the
damping term is present, i.e.~if $   \partial_t m = -m \times( m \times \Delta
m)$, it is called a harmonic map heat flow into $\mathbb{S}^2$ \cite{gustafson2010asymptotic}.

In 1935, Landau and Lifshitz \cite{landau1935theory} calculated the
structure of the domain walls between antiparallel domains, which
started the theory of micromagnetics. The theory was further developed
by W. F. Brown Jr in \cite{brown1962magnetostatic}.  Applications of
micromagnetics include magnetic sensor technology, magnetic recording,
and magnetic storage devices such as hard drives and magnetic
memory (MRAM) \cite{fidler2000micromagnetic}.

Finite difference methods for micromagnetics can be derived in
  two different ways \cite{miltat2007numerical}.  The first is
a field-based approach in which the effective field $h$
  is discretized directly, and the other is an energy-based
approach in which the effective field is derived from the
discretized energy. Finite element methods can also be derived
  in a number of ways. In \cite{schrefl1999finite}, the Landau-Lifshitz-Gilbert equation (\ref{eq:LLG}) is used to 
obtain a discrete system by approximating the magnetization by
piecewise linear function on a finite element mesh and then applying
time integration in the resulting system of ODEs. In
\cite{fidler2000micromagnetic}, the effective field $h$ is
calculated by taking a functional derivative of the discretized
energy, where the magnetization in the energy formula is approximated by
piecewise linear functions. Extensive work has also been
  done developing time stepping schemes for micromagnetics
\cite{cimrak2007survey,garcia2007numerical}. In
\cite{jiang2001hysteresis}, semi-analytic integration in time was
introduced, which is explicit and first order in time and allows
stepsize control. In
\cite{krishnaprasad2001cayley,lewis2003geometric}, geometric
integration methods were applied to the Landau-Lifshitz equation. In
\cite{wang2001gauss,garcia2003improved}, a Gauss-Seidel projection
method was developed that treats the gyromagnetic term and damping
term separately to overcome the difficulty of the stiffness and the
nonlinearity of the equation.

However, relatively little work has been done deriving error
  estimates or establishing rigorous convergence results for weak
  solutions.  In a series of papers
\cite{alouges2006convergence,alouges2008new,alouges2012convergent},
Alouges and various co-authors introduced a convergent finite element
method based on equation (\ref{eq:LLGdifferent}), an equivalent
  form of the Landau-Lifshitz equation, that is first order in time.
The idea is to use a tangent plane formulation at each timestep, where
the velocity vector lies in the finite element space perpendicular to
the magnetization at each node. One advantage of this method is that
at each step, only a linear system has to be solved, although the
Landau-Lifshitz-Gilbert equation is nonlinear.  More recently, they
developed a formally second order in time scheme
\cite{kritsikis2014beyond, alouges2014convergent} that performs better
than first order in practice, though not fully at second order.
Another finite element scheme was introduced by Bartels and Prohl in
\cite{bartels2006convergence} based on the Landau-Lifshitz-Gilbert
equation, which is an implicit, unconditionally stable method, but
involves solving nonlinear equations at each timestep.  This
  method is second order in time; however, there is still a time step
  constraint, namely that $k/h^2$ remain bounded, to guarantee
  the existence of the solutions of the nonlinear systems.  Cimr{\'a}k
  \cite{cimrak2009convergence} introduced a finite element method
  based on the Landau-Lifshitz equation, which is an implicit,
  unconditionally stable method, but also has nonlinear inner
  iterations. We note that the Backward Euler method and
  higher-order diagonally implicit Runge-Kutta (DIRK) methods
  \cite{hairer:II} generally involve solving nonlinear equations at
  each internal Runge-Kutta stage when applied to nonlinear PDEs.

In this article, we introduce a family of mass-lumped 
finite element methods for the Landau-Lifshitz equation.
The implicit version is similar in computational complexity to
  the algorithms in \cite{alouges2006convergence,alouges2008new,alouges2012convergent}
  in that each timestep involves solving a large sparse linear
  system. The explicit method is more efficient than the explicit
  version of \cite{alouges2006convergence,alouges2008new,alouges2012convergent} as it
  is completely explicit --- it does not even require that a linear
  system involving a mass matrix be solved as the
 effective mass matrix is diagonal.
   The method
 involves finding the velocity vector in the tangent
plane of the magnetization by discretizing the Landau-Lifshitz
equation instead of the Landau-Lifshitz-Gilbert equation, as
was done in
\cite{alouges2006convergence,alouges2008new,alouges2012convergent,
  kritsikis2014beyond, alouges2014convergent}.  By building a
numerical scheme based on the Landau-Lifshitz equation instead
of the Landau-Lifshitz-Gilbert equation, we can naturally apply
the scheme to  limiting cases such as the
Schr\"{o}dinger map or harmonic map heat flow
 \cite{bartels2007constraint,de2009numerical, gustafson2010asymptotic, komineas2007rotating,komineas2015skyrmion}.
The main result of the paper is a proof
that the numerical solution of our scheme for the
  Landau-Lifshitz equation converges to a weak solution of the
Landau-Lifshitz-Gilbert equation, using a simple technique that
cancels out the product of two weakly convergent sequences. Our proof
builds on tools developed in \cite{alouges2012convergent}.
For simplicity,
  we defer the treatment of the stray field to future work. This
  term poses computational challenges
   \cite{blue1991using,fredkin1990hybrid,yuan1992fast,brunotte1992finite,tsukerman1998multigrid,popovic2005applications,long2006fast,garcia2006adaptive, livshitz2009nonuniform,exl2012fast,abert2012fast,abert2013numerical},
   but does not affect the
  convergence results since it is a lower order term in comparison
  to the exchange term; see \cite{alouges2012convergent}.

The paper is organized as follows. In section \ref{sec:mesh}, we
introduce a finite element mesh and review the weak formulation of the
Landau-Lifshitz-Gilbert equation. In section \ref{sec:scheme}, the
main algorithm and the main theorem will be introduced. In section
\ref{sec:numerical}, we conduct a numerical test for the equation $h=
\Delta m$ on the unit square with periodic boundary conditions, where
an exact analytical solution is known from \cite{fuwa2012finite}.
In section \ref{sec:pf}, the proof of the main theorem will be
presented.

\section{Weak solutions, meshes and the finite element space}
\label{sec:mesh}
Let us denote $\Omega_T= \Omega \times (0,T)$.
The definition of a weak solution of the Landau-Lifshitz-Gilbert
equation is given by

\vspace*{5pt}
\begin{definition} \label{def:weak}
 Let $m_0(x) \in H^1(\Omega)^3$ with $|m_0(x)| = 1$ $a.e.$ Then $m$ is
 a weak solution of (\ref{eq:LLG}) if for all $T>0$,
\begin{itemize}
\item[(i)] $m(x,t) \in H^1(\Omega_T)^3$, $|m(x,t)|=1$ $a.e.$,
\item[(ii)] $m(x, 0)= m_0(x)$ in the trace sense,
\item[(iii)] $m$ satisfies
  \end{itemize}
\begin{align}\label{eq:defLLG}
  \int_{\Omega_T} &\partial_t m \cdot \phi -
  \alpha \int_{\Omega_T} (m \times\partial_t m )\cdot \phi\\ \notag
  &= (1+\alpha^2) \eta\sum_{l=1}^d\int_{\Omega_T}
  (m \times \partial_{x_l} m) \cdot \partial_{x_l} \phi   - (1+\alpha^2) \sum_{l=1}^d\int_{\Omega_T}
  (m \times \bar h(m)) \cdot  \phi .
\end{align}
for all $\phi \in  H^1(\Omega_T)^3$.
\begin{itemize}
\item[(iv)]  $m$ satisfies an energy inequality
\begin{equation} \label{eq:energyiinequalitydef}
  C  \int_{\Omega_T}  |\partial_t m|^2  +   \mc{E}(m(x,T)) \leq  \mc{E}(m(x,0)).
  \end{equation}
for some constant $C>0$, where the energy $\mc{E}(m)$ is defined in equation (\ref{eq:LLenergydef}).
\end{itemize}
\end{definition}

\vspace*{5pt}
\noindent
In \cite{alouges2006convergence,bartels2006convergence}, the value of
$C$ in ($iv$) is taken to be $C=\frac{\alpha}{1+\alpha^2}$.
The existence of global
weak solution of the Landau-Lifshitz equation in $\Omega \subset
\mathbb{R}^3$ into $\mathbb{S}^2$ was proved in
\cite{alouges1992global,guo1993landau}. The nonuniqueness of weak
solution was proved in \cite{alouges1992global}. 

Let the domain $\Omega \subset \mbb{R}^d$ where $d=2$ or $3$
be discretized into triangular or tetrahedral
elements $\{ \mathcal{T}_h\}_h$ of mesh size at most $h$ with vertices
$(x_i)_{i=1}^N$.  Let the family of partitions $\mathcal{T}= \{
\mathcal{T}_h\}_h $ be admissible, shape regular and uniform \cite{braess2007finite}.
 Let
$\{\phi_i\}_{1 \leq i \leq N}$ be piecewise linear nodal basis
functions for $\mathcal{T}$, such that $\phi_i(x_j)=\delta_{ij}$,
where $\delta_{ij}$ is a Kronecker delta function. We will consider a
vector-valued finite element space $F^h$ defined by
$F^h = \{ w^h \; | \; w^h(x) = \sum_{i=1}^N w^h_i \phi_i(x), \; w^h_i \in \mathbb{R}^3 \}.$
%
The discrete magnetization $m^h$ is required to belong to the
submanifold $M^h$ of $F^h$ defined by
$M^h = \{ m^h \in F^h \; | \; m^h(x)= \sum_{i=1}^N m^h_i \phi_i(x), \; |m^h_i|=1 \}.$
We define the interpolation operator $I_h: C^0(\Omega,
\mathbb{R}^3) \to F^h $ by
\begin{equation}\label{eq:interpolation}
I_h(m)=\sum_{i=1}^N m(x_i) \phi_i(x).
\end{equation} 
We need the following
additional conditions for our finite element method: There exist some
constants $C_1, C_2, C_3, C_4>0$ such that
\begin{equation}\label{eq:mesh}
\begin{aligned}
& C_1h^d \leq b_i = \int_\Omega \phi_i  \leq C_2 h^d, \\
& \left|M_{ij} \right| = \left|\int_\Omega \phi_i \phi_j \right| \leq C_3 h^d, \\
& \left|\partial_{x_l} \phi_i \right| \leq \frac{C_4}{h}\\
&\ \int_\Omega \nabla\phi_i \cdot \nabla\phi_j  \leq 0, \quad \text{ for } i \neq j,
\end{aligned}
\end{equation}
for all $h>0$, $i, j=1, \dots, N$ and $l=1, \dots, d$.

\section{The finite element scheme, the algorithm, and the main theorem}
\label{sec:scheme}
To illustrate how we obtain Algorithm 1 below, consider the simple
case in which only the exchange energy term is present in the
  effective field, i.e.~$h= \eta \triangle m$ from
(\ref{eq:h:def}).
Let's first
consider the weak form of the Landau-Lifshitz equation with $h=\eta
\triangle m$,
\begin{equation}
\begin{aligned}\label{eq:defLL}
  \int_{\Omega_T} \partial_t m \cdot w& =  \eta \sum_{l=1}^d\int_{\Omega_T}
  (m \times \partial_{x_l} m) \cdot \partial_{x_l} w \\
&- \alpha \eta \sum_{l=1}^d  \int_{\Omega_T}   \partial_{x_l} m \cdot  \partial_{x_l} w 
 +\alpha \eta \sum_{l=1}^d \int_{\Omega_T} (  \partial_{x_l}  m \cdot  \partial_{x_l}  m)(m \cdot w)  .
 \end{aligned}
 \end{equation}
Taking this weak form
as a hint, we would like to find $v =
\sum_{j=1}^N v_j \phi_j \in F^h$ such that
\begin{equation} \label{eq:approxLL}
\begin{aligned}
&  \int_\Omega \sum_{j=1}^N v_j \phi_j \cdot  w_i \phi_i  =\eta \sum_{l=1}^d \sum_{j=1}^N \int_\Omega
  (m_i \times m_j \partial_{x_l} \phi_j ) \cdot  \partial_{x_l} \phi_i \; w_i \\
  &\hspace{0.5cm} - \alpha\eta \sum_{l=1}^d \sum_{j=1}^N \int_\Omega m_j \partial_{x_l}\phi_j \cdot  \partial_{x_l} \phi_i w_i 
   + \alpha\eta\sum_{l=1}^d \sum_{j=1}^N  \int_\Omega ( \partial_{x_l} \phi_j m_j \cdot m_i)(m_i \cdot  w_i\partial_{x_l} \phi_i )
\end{aligned}
\end{equation}
for $i=1, \dots, N$, where $m = \sum_{j=1}^N m_j \phi_j(x) \in M^h$,
$w \in (C^\infty(\Omega))^3$ and $w_i =I_h(w)(x_i)= w(x_i)$. Then, with $w_i$ as
$(1,0,0)$, $(0,1,0)$ or $(0,0,1)$ in equation (\ref{eq:approxLL}), we
obtain
\begin{equation} \label{eq:Mv}
\begin{aligned}
(\mathbf{M} v)_i = \eta\; m_i \times (\mathbf{A}m)_i + \alpha \eta \; m_i \times (m_i \times (\mathbf{A}m)_i)
\end{aligned}
\end{equation}
for $i=1, \dots, N$, where 
$\mathbf{M}= \begin{bmatrix} M & 0 & 0 \\ 0 & M & 0 \\ 0 & 0 & M\end{bmatrix}$
and
$\mathbf{A}= \begin{bmatrix} A & 0 & 0 \\ 0 & A & 0 \\ 0 & 0 & A\end{bmatrix}$
are  $3N \times 3N$  block diagonal matrices with each block $M$ and $A$ a mass or stiffness matrix,
  i.e.~$M_{ij}= \int_\Omega \phi_i \phi_j $, and
 $A_{ij} = \sum_{l=1}^d \int_\Omega \partial_{x_l}
    \phi_i \partial_{x_l} \phi_j $. 
    Note that $m_i \cdot (\mathbf{M} v)_i=0$,
    so approximating $v$ by $\hat{v} = \frac{\mb{M} v}{b}$ yields a tangent vector to the constraint manifold $M^h$, where $b_i =
\int_\Omega \phi_i $.
    The left hand side of (\ref{eq:Mv}) is then $b_i \hat{v}_i$ which is a mass-lumping approximation.
    This suggests the following algorithm.

    \vspace*{5pt}
\textbf{Algorithm 1.} For a given time $\bar T>0$, set $J=[\frac{\bar T}{k}]$.
\begin{enumerate} \label{algorithm}
\item Set an initial discrete magnetization $m^0$ at the nodes
  of the finite element mesh described in \S~\ref{sec:mesh} above.
\item For $j=0, \dots, J$,
\item [  a.]  compute a velocity vector $\hat{v}^j_i$ at each node by
\begin{equation} \label{eq:algorithm}
\begin{aligned}
  \hat{v}^j_i = \frac{(\mathbf{M}v^j)_i}{b_i}  = \frac{  \eta \;m^j_i
    \times (\mathbf{A}m+\theta k \mathbf{A} \hat v)_i^j + \alpha \eta \;
    m^j_i \times (m_i^j \times (\mathbf{A}m+\theta k \mathbf{A} \hat v)_i^j)}{b_i} \\
  - \frac{ m^j_i \times (\mathbf{M} \bar h (m)+\theta k \mathbf{M} \bar h (\hat v))^j_i
    + \alpha m^j_i \times (m^j_i \times (\mathbf{M}\bar h (m) +
    \theta k \mathbf{M} \bar h (\hat v))_i^j)}{b_i} .
\end{aligned}
\end{equation}
for $\theta \in [0,1]$ and for $i=1, \dots, N$.  
\item[ b.] Compute $m_i^{j+1} = \frac{m_i^j + k \hat v_i^j}{|m_i^j + k
  \hat v_i^j|}$ for $i=1, \dots, N$.
\end{enumerate}

We define the
time-interpolated magnetization and velocity as in
\cite{alouges2012convergent}:
\begin{definition} \label{def:main}
 For $(x,t) \in \Omega \times [jk,(j+1)k) \subset\Omega \times [0,T)$, where $T= Jk$, define
\begin{align*}
& m^{h,k}(x,t) = m^j(x),\\
&\bar m^{h,k}(x,t) = \frac{t-jk}{k} m^{j+1}(x) + \frac{(j+1)k-t}{k} m^j(x),\\
&\hat v^{h,k}(x,t) = \hat v^j(x), \\
& v^{h,k}(x,t) = v^j(x).
\end{align*}
\end{definition}
The main theorem in this article is the following theorem,
which is proved in section \ref{sec:pf}.
\begin{theorem}\label{thm:main}
Let $m_0\in H^1(\Omega, \mathbb{S}^2)$ and suppose $m_0^h \to m_0$ in
$H^1(\Omega)$ as $h \to 0$.
 Let $\theta \in [0,1]$, and for $0 \leq \theta <\frac{1}{2}$,  assume  that $\frac{k}{h^2} \leq C_0$, for some $C_0>0$. 
  If the triangulation $\mathcal{T}= \{
\mathcal{T}_h\}_h $ satisfies condition (\ref{eq:mesh}), then the
sequence $\{m^{h,k}\}$, constructed by Algorithm 1 and definition \ref{def:main}, has a subsequence that converges weakly to a
weak solution of the Landau-Lifshitz equation.
\end{theorem}

\section{Numerical Results}\label{sec:numerical}

Before proving Theorem \ref{thm:main}, we demonstrate the effectiveness of the scheme on a test problem.
We conduct a numerical experiment for the Landau-Lifshitz equation
  (\ref{eq:LL}) with effective field involving only the exchange energy term, with
  $h= \Delta m$ in equation (\ref{eq:h:def}), on the unit square with
  periodic boundary conditions. This corresponds to setting $\eta = 1$
  and $\bar h =0$ in equation (\ref{eq:algorithm}) in Algorithm 1.
For the convergence study, we used an explicit method $(\theta=0)$ and
an implicit method $(\theta =0.5) $ on a structured and unstructured
mesh.  The unstructured mesh with point and line sources, which is an arbitrary mesh,
was generated 
using DistMesh \cite{persson2004simple}, with
an example shown in figure \ref{figure:unstructured}.

\begin{figure}[t]
  \begin{center}
   \includegraphics[width=.6\linewidth,trim=20 5 50 10,clip]{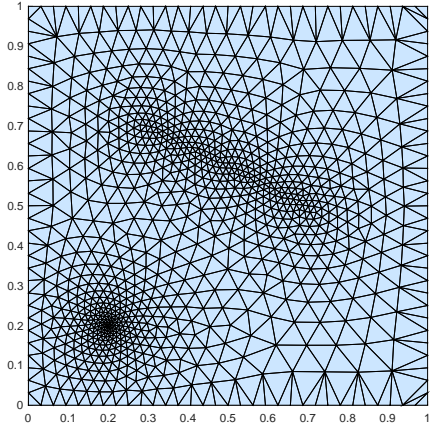}
  \end{center}
  \caption{\label{figure:unstructured}
    Unstructured mesh with point and line sources, with $h=1/32$.}
\end{figure}
 The $L^\infty$
and $L^2$ errors were measured relative to an exact solution for the Landau-Lifshitz
 equation with $h=  \Delta m$ 
from \cite{fuwa2012finite}, namely
\begin{equation} \label{eq:anal}
\begin{aligned}
&m^x(x_1,x_2,t) =\frac{1}{d(t)} \sin \beta \cos(k(x_1+x_2)+g(t)), \\
&m^y(x_1,x_2,t) =\frac{1}{d(t)} \sin \beta \sin(k(x_1+x_2)+g(t)), \\
&m^z(x_1,x_2,t) =\frac{1}{d(t)} e^{2 k^2 \alpha t} \cos \beta.
\end{aligned}
\end{equation}
Here $\beta = \frac{\pi}{24}$, $k=2\pi$, $d(t) = \sqrt{\sin^2\beta +
  e^{4k^2\alpha t} \cos^2 \beta}$ and $g(t)=\frac{1}{\alpha}
\log(\frac{d(t)+e^{2k^2 \alpha t} \cos \beta}{1+\cos \beta})$.  The
numerical results are summarized in the tables
\ref{table:structured:explicit} and
\ref{table:structured:implicit}.
Figure \ref{figure:l2} shows the convergence rate of the methods, which
is first order in the time step $k$ and second order in the
mesh size $h$.

\begin{figure}[h]
   \includegraphics[width=.495\linewidth,trim=20 5 50 10,clip]{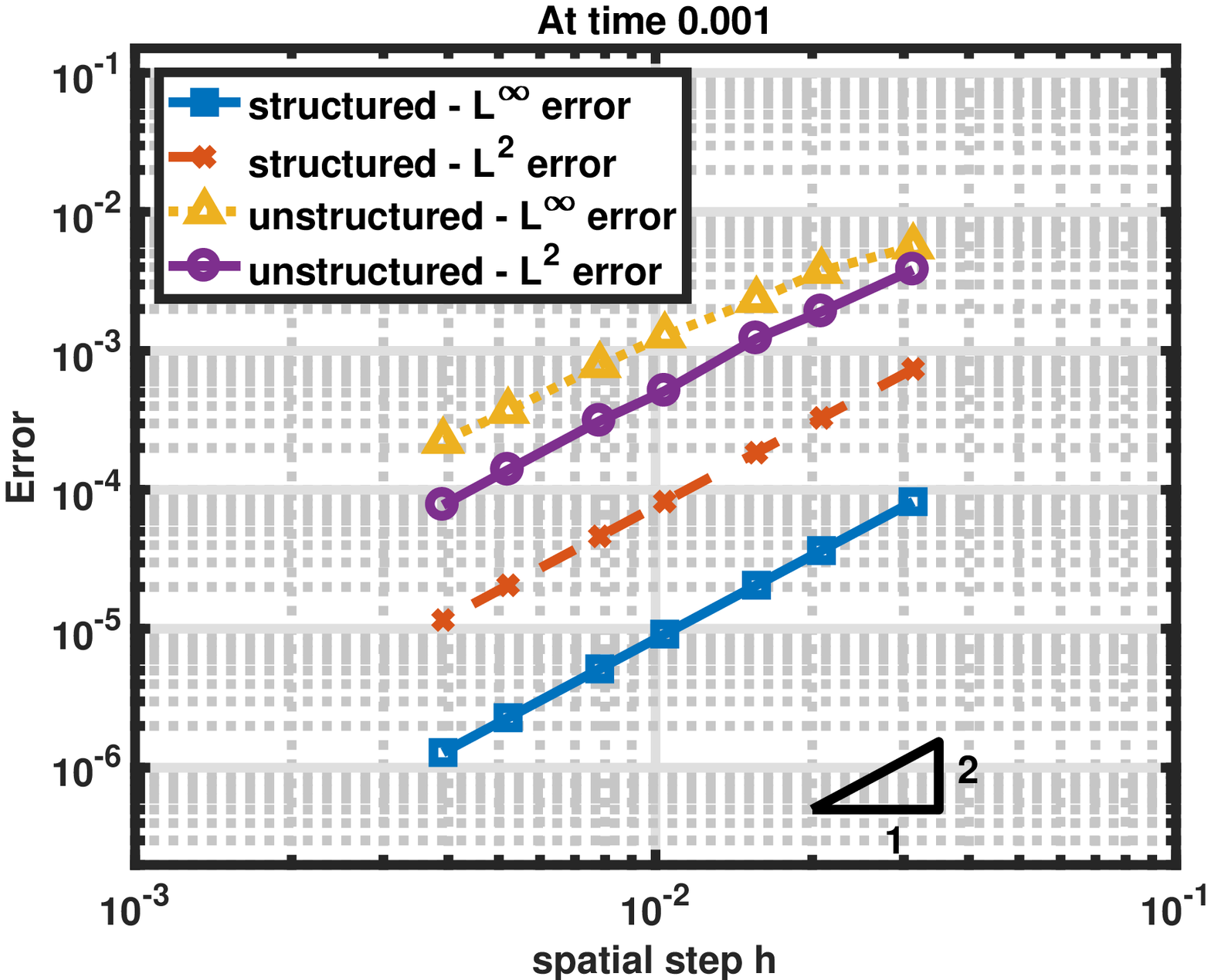}
   \hfill
   \includegraphics[width=.495\linewidth,trim=20 5 50 10,clip]{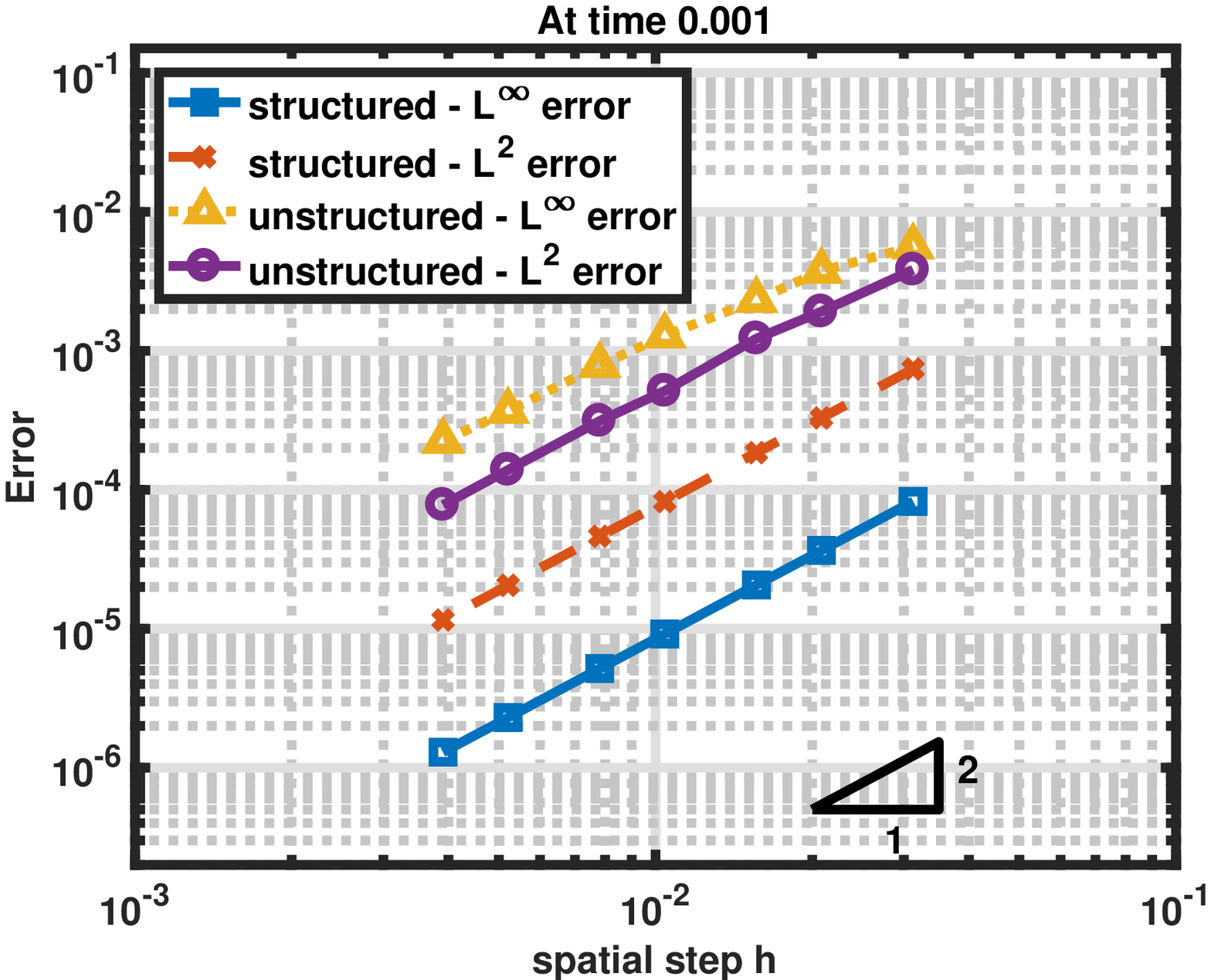}
 \caption{Convergence plot, Left :  Explicit method, Right: Implicit method. }
  \label{figure:l2}
\end{figure}

 \begin{table}[h]
\centering
 \begin{tabular}{|c| c| c|c|c|c| c|c|c|}
   \hline
 & \multicolumn{4}{|c|}{Structured mesh}  & \multicolumn{4}{|c|}{Unstructured mesh}\\  
 \hline
 $\frac{1}{h}$ 	&  \scalebox{.8}[1.0]{$|| m-m^h||_{L^\infty}$} & rate & \scalebox{.8}[1.0]{$|| m-m^h||_{L^2}$}& rate  & \scalebox{.8}[1.0] {$|| m-m^h||_{L^\infty}$} & rate  & \scalebox{.8}[1.0]{$|| m-m^h||_{L^2}$}& rate  \\
\hline
32	&	8.22e-05	&	 2.00	&	7.40e-04	&	 2.00	&	5.61e-03	&	 1.28	&	3.83e-03	&	 1.65	 \\ 
64	&	2.06e-05	&	 2.00	&	1.85e-04	&	 2.00	&	2.32e-03	&	 1.56	&	1.22e-03	&	 1.97	 \\ 
128	&	5.15e-06	&	 2.00	&	4.63e-05	&	 2.00	&	7.87e-04	&	 1.81	&	3.13e-04	&	 2.01	 \\ 
256	&	1.29e-06	&	     	&	1.16e-05	&	     	&	2.25e-04	&	     	&	7.77e-05	&	     	 \\ 
 \hline
  \end{tabular}
 \caption{{\bf Explicit method ($\theta=0$) }: $L^\infty$ and $L^2$
   error and convergence rates on a structured and unstructured mesh with spatial step
   $h$, time step $k=8\cdot10^{-5} h^2$ and time $0.001$.
 }
   \label{table:structured:explicit}
\end{table}

\begin{table}[h]
\centering
 \begin{tabular}{|c| c| c|c|c|c| c|c|c|}
   \hline
 & \multicolumn{4}{|c|}{Structured mesh}  & \multicolumn{4}{|c|}{Unstructured mesh}\\  
 \hline
 $\frac{1}{h}$ 	&  \scalebox{.8}[1.0]{$|| m-m^h||_{L^\infty}$} & rate & \scalebox{.8}[1.0]{$|| m-m^h||_{L^2}$}& rate  & \scalebox{.8}[1.0] {$|| m-m^h||_{L^\infty}$} & rate  & \scalebox{.8}[1.0]{$|| m-m^h||_{L^2}$}& rate  \\
\hline
32	&	8.26e-05	&	 2.00	&	7.40e-04	&	 2.00	&	5.61e-03	&	 1.28	&	3.83e-03	&	 1.65	 \\ 
64	&	2.07e-05	&	 2.00	&	1.85e-04	&	 2.00	&	2.32e-03	&	 1.56	&	1.22e-03	&	 1.97	 \\ 
128	&	5.17e-06	&	 2.00	&	4.63e-05	&	 2.00	&	7.87e-04	&	 1.81	&	3.13e-04	&	 2.01	 \\ 
256	&	1.29e-06	&	     	&	1.16e-05	&	     	&	2.25e-04	&	     	&	7.77e-05	&	     	 \\ 
 \hline
  \end{tabular}
  \caption{ {\bf Implicit method $(\theta=\frac{1}{2})$ }: $L^\infty$
    and $L^2$ error and convergence rates on a structured and unstructured mesh , with
    spatial step $h$, time step $k=0.00256 h^2$ and time $0.001$.}
\label{table:structured:implicit}
\end{table}

\subsection{Going beyond first order in time}
In this section, we propose a method which is second order in
  time, by replacing the nonlinear projection step 2 (b) in Algorithm
  1 by a linear projection step, and test the convergence order.  In
  Algorithm 1, step 2 (a) can be viewed as the predictor step and 2
  (b) as the corrector step.  The corrector step was used to conserve
  the length of the magnetization at each node.  By replacing this
  nonlinear projection by a linear projection step, it not only
  preserves the length of the magnetization, but also makes the method
  higher order.  Moreover, it has a similar complexity as the
  nonlinear projection step in that you only need to solve
  a $3 \times 3$ matrix equation for each node.  We
  defer a rigorous analysis to future work and present here the
  modified algorithm and some convergence test results.

     \vspace*{5pt}
 \textbf{Algorithm 2.} For a given time $\bar T>0$, set $J=[\frac{\bar T}{k}]$.
\begin{enumerate} \label{algorithm2}
\item Set an initial discrete magnetization $m^0$ at the nodes
  of the finite element mesh described in \S~\ref{sec:mesh} above.
\item For $j=0, \dots, J$,
\item [  a.]  compute an intermediate magnetization vector $m^*_i$ at each node by
\begin{equation*} \label{eq:algorithm2}
\begin{aligned}
& {m_i^*- m_i^j \over k} =  \hat{v}^j_i = \frac{(\mathbf{M}v^j)_i}{b_i}  \\
& = \frac{  \eta \;m^j_i
    \times (\mathbf{A}m+\theta k \mathbf{A} \hat v)^j_i + \alpha \eta \;
    m^j_i \times (m_i^j \times (\mathbf{A}m+\theta k \mathbf{A} \hat v)_i)^j}{b_i} \\
&  - \frac{ m^j_i \times (\mathbf{M} \bar h (m)+\theta k \mathbf{M} \bar h (\hat v))^j_i
    + \alpha m^j_i \times (m^j_i \times (\mathbf{M}\bar h (m) +
    \theta k \mathbf{M} \bar h (\hat v))_i)^j}{b_i} .
\end{aligned}
\end{equation*}
for $\theta \in [0,1]$ and for $i=1, \dots, N$.  
\item[ b.] Compute $m_i^{j+1} $ for $i=1, \dots, N$.
  \begin{equation*} \label{eq:algorithm3}
\begin{aligned}
& {m_i^{j+1}- m_i^j \over k} =  \\
& = \frac{  \eta \; {m_i^{j+1}+m_i^j \over 2}
    \times (\mathbf{A}m^{j+1/2}_i ) + \alpha \eta \;
    {m_i^{j+1}+m_i^j \over 2} \times (m_i^{j+1/2} \times (\mathbf{A}m^{j+1/2} )}{b_i} \\
&  - \frac{ {m_i^{j+1}+m_i^j \over 2} \times \mathbf{M} \bar h (m^{j+1/2})_i
    + \alpha {m_i^{j+1}+m_i^j \over 2} \times (m^{j+1/2}_i \times (\mathbf{M}\bar h (m^{j+1/2 }_i) )}{b_i} .
\end{aligned}
\end{equation*} where $m^{j+1/2} = \frac{m^{j}+m^*}{2}$ for $i=1, \dots, N$.
\end{enumerate}

\vspace*{5pt}  As before, we conduct a numerical test for the
  Landau-Lifshitz equation (\ref{eq:LL}) with effective field
  involving only the exchange energy term, with $h= \Delta m$ in
  equation (\ref{eq:h:def}), on the unit square with periodic boundary
  conditions, to compare the two algorithms.  For the convergence
  study, we used an implicit method $(\theta =0.5) $ on a structured
  and unstructured mesh. One of the unstructured meshes was shown in
  figure \ref{figure:unstructured}.  The $L^2$ errors were measured
  relative to an analytical solution (\ref{eq:anal}) for the
  Landau-Lifshitz equation with $h= \Delta m$.  The numerical results
  are summarized in the tables \ref{table:structured:implicithigh}.
  Figure \ref{figure:p3} shows the convergence rates of the methods,
  which shows first order in $k$ for Algorithm 1 and second order
  convergence for Algorithm 2.  
\begin{figure}[h]
   \includegraphics[width=.495\linewidth,trim=20 5 50 10,clip]{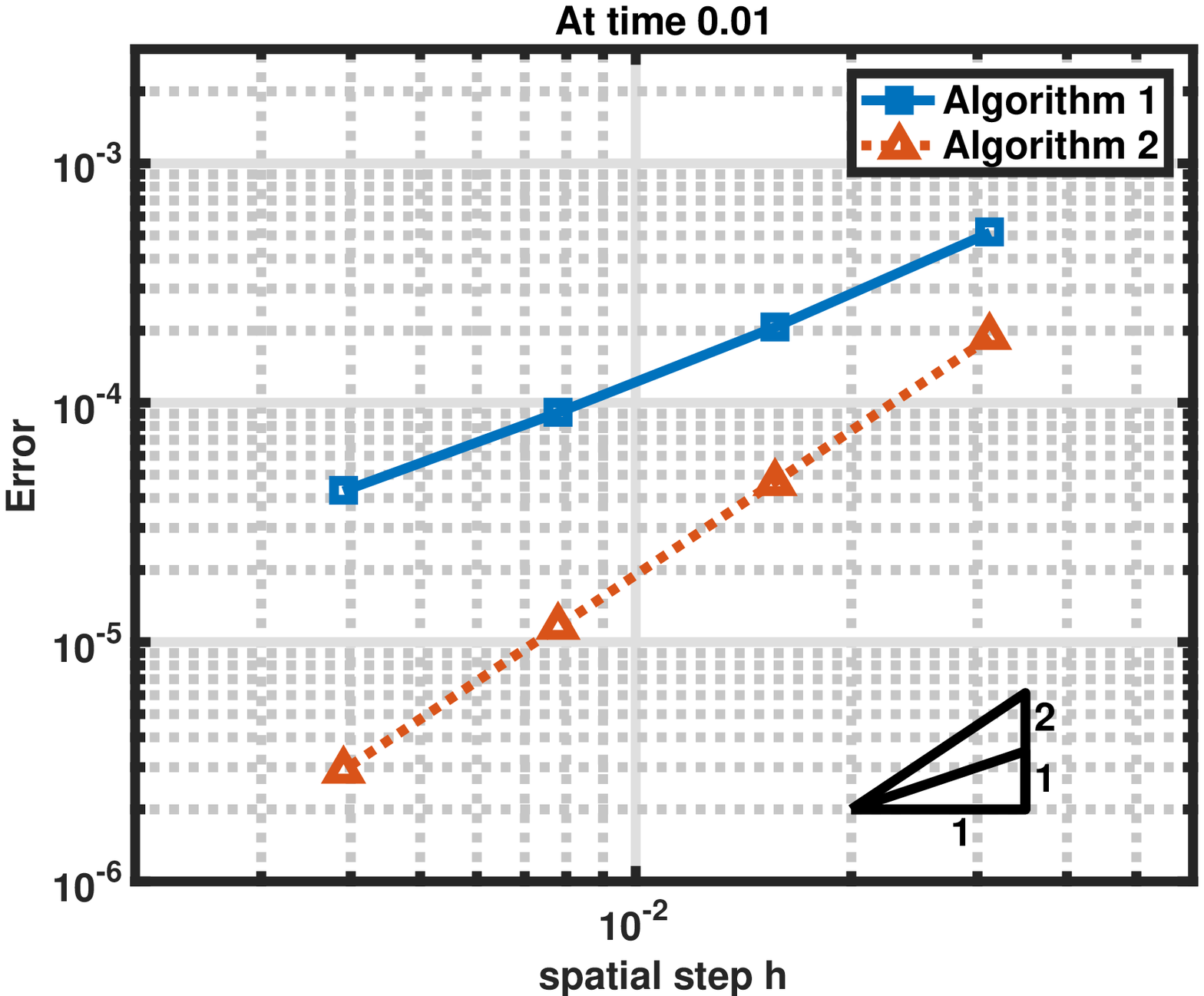}
   \hfill
   \includegraphics[width=.495\linewidth,trim=20 5 50 10,clip]{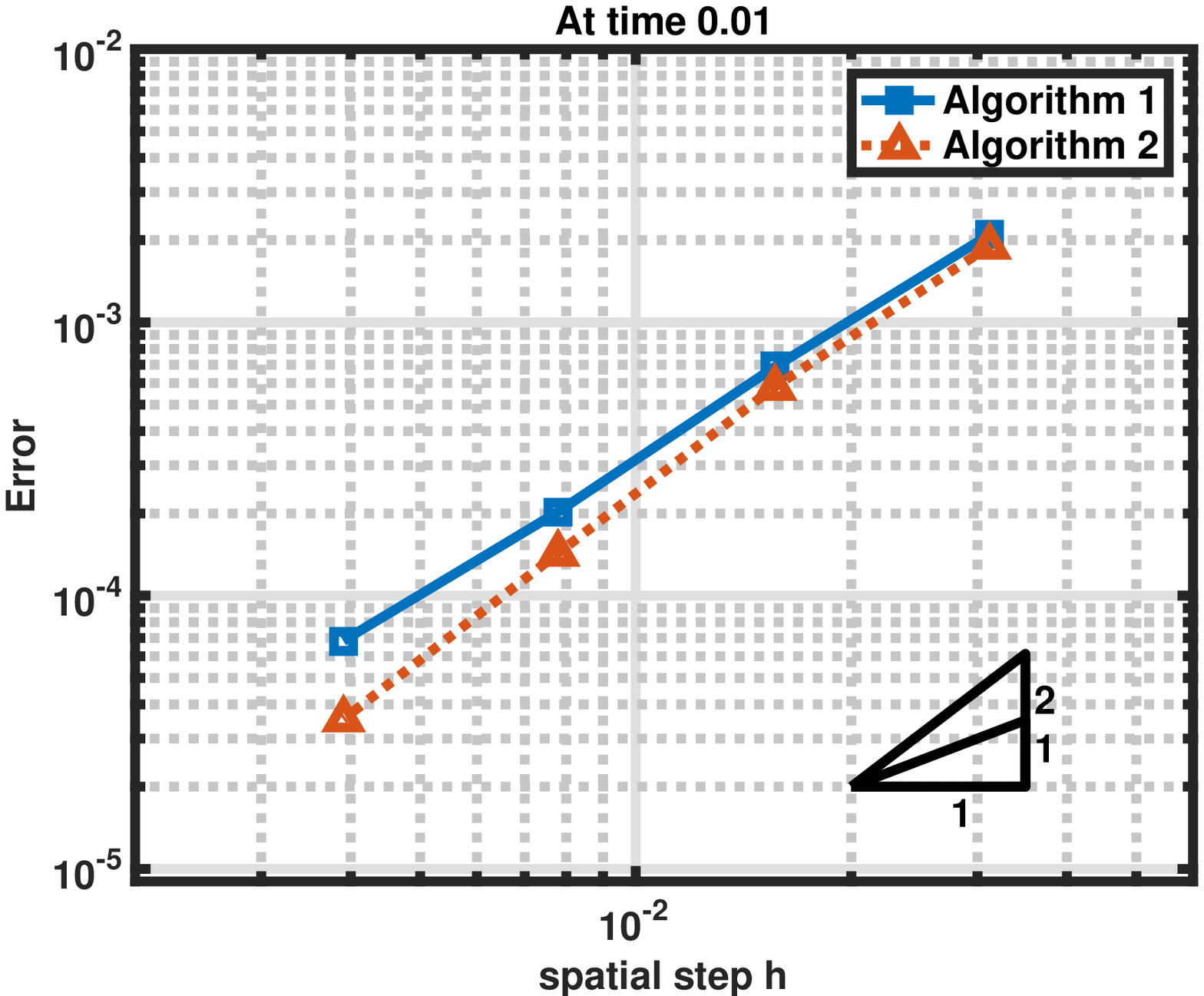}
 \caption{Convergence plot, Left :  Structured mesh, Right : Unstructured mesh. }
  \label{figure:p3}
\end{figure}

 \begin{table}[h]
\centering
 \begin{tabular}{|c| c| c|c|c|c| c|c|c|}
   \hline
 & \multicolumn{4}{|c|}{Structured mesh}  & \multicolumn{4}{|c|}{Unstructured mesh}\\  
 \hline
  $\frac{1}{h}$ 	&  Alg. 1&  rate  &  Alg. 2 &  rate  &   Alg. 1&  rate  &   Alg. 2 &  rate \\
\hline
32	&	5.14e-04	&	 1.32	&	1.87e-04	&	 2.01	&	2.10e-03	&	 1.61	&	1.91e-03	&	 1.72	 \\ 
64	&	2.06e-04	&	 1.18	&	4.66e-05	&	 2.00	&	6.87e-04	&	 1.75	&	5.81e-04	&	 2.01	 \\ 
128	&	9.12e-05	&	 1.10	&	1.16e-05	&	 2.00	&	2.04e-04	&	 1.57	&	1.44e-04	&	 2.03	 \\ 
256	&	4.27e-05	&	     	&	2.91e-06	&	     	&	6.84e-05	&	     	&	3.53e-05	&	     	 \\ 
 \hline
  \end{tabular}
  \caption{ {\bf Implicit method $(\theta=\frac{1}{2})$ }: 
 $L^2$ error and convergence rates on a structured and unstructured mesh , with
    spatial step $h$, time step $k=0.04 h$ and time $0.01$.}
     \label{table:structured:implicithigh}

\end{table}

\section{Proof of Theorem \ref{thm:main}}
\label{sec:pf}
In this section, we present the proof of the theorem, which states
that the sequence $\{m^{h,k}\}$, constructed by Algorithm 1 and
Definition \ref{def:main}, has a subsequence that converges weakly to
a weak solution $m$ of the Landau-Lifshitz-Gilbert equation under some
conditions.  That is, we show that the limit $m$ satisfies Definition
\ref{def:weak}.  In section \ref{subsection:pf}, we derive a
discretization of the weak form of the Landau-Lifshitz-Gilbert
equation satisfied by the $\{m^{h,k} \}$, namely
(\ref{eq:approxLLx3}).  In section \ref{subsec:energy}, we derive
energy estimates to show that the sequences $m^{h,k}$, $\bar m^{h,k}$
and $\hat v^{h,k}$ converge to $m$ in a certain sense made precise in
section \ref{subsection:weak}.  In section \ref{subsection:detail}, we
show that each term of the discretization of the weak form converges
to the appropriate limit, so that the limit $m$ satisfies the weak
form of the Landau-Lifshitz-Gilbert equation.  In section
\ref{subsection:energy}, we show that the limit $m$ satisfies the
energy inequality (\ref{eq:energyiinequalitydef}) in Definition
\ref{def:weak} (iv).  Finally, in section
  \ref{subsection:magnitude}, we establish that the magnitude
  of $m$ is $1$ a.e. in $\Omega_T$.

\subsection{Equations that $m^{h,k}$ and $ v^{h,k}$ satisfy}
\label{subsection:pf}
In this section, we derive
a discretization of the weak
  form of the Landau-Lifshitz-Gilbert equation. 
  This form is easier to use for the proof of Theorem \ref{thm:main}, since it does not involve the product
  of the weakly convergent sequences.  
  In general, a product of weakly convergent sequences is not weakly convergent. It is convergent
  only in some certain cases, such as when the sequences satisfy the
  hypothesis of the div-curl lemma \cite{tartar2009general,
    christiansen2005div}.

The generalized version of equation (\ref{eq:approxLL}) including
  all the terms in the effective field $h$ 
  (\ref{eq:h:def}) and with $0\leq \theta \leq 1$ is 
  \begin{equation} \label{eq:approxLLx}
\begin{aligned}
& \int_\Omega  v^{h,k}  \cdot w^h  = \eta \sum_{l,i} \int_\Omega (m^{h,k}_i \times \partial_{x_l} (m^{h,k}+\theta k \hat v^{h,k})  ) \cdot  \partial_{x_l} \phi_i \; w^h_i  \\
&\hspace{2cm} - \alpha\eta \sum_{l,i}   \int_\Omega \partial_{x_l} (m^{h,k}+\theta k \hat v^{h,k})  \cdot  \partial_{x_l} \phi_i w^h_i \\
&\hspace{2cm} + \alpha\eta\sum_{l,i}   \int_\Omega ( \partial_{x_l}  (m^{h,k}+\theta k \hat v^{h,k}) \cdot m^{h,k}_i)(m^{h,k}_i \cdot  w^h_i )\partial_{x_l} \phi_i \\
 &\hspace{2cm} - \sum_i  \int_\Omega (m_i^{h,k} \times \bar h (m^{h,k}+\theta k \hat v^{h,k})  ) \cdot   \phi_i  w^h_i \\
&\hspace{2cm}  + \alpha \sum_i  \int_\Omega\bar h (m^{h,k}+\theta k \hat v^{h,k})  \cdot  \phi_i w^h_i \\
&\hspace{2cm}- \alpha \sum_i   \int_\Omega (  \bar h (m^{h,k}+\theta k \hat v^{h,k}) \cdot m_i)(m_i \cdot  w^h_i \phi_i ).
\end{aligned}
\end{equation}
 In fact, by taking $w^h_i$ as $(1,0,0),
  (0,1,0)$ and $(0,0,1)$ in 
  (\ref{eq:approxLLx}), we get  (\ref{eq:algorithm}) in
  Algorithm 1. 
Setting $w^h = \sum_{j=1}^N (m^{h,k}_j \times u^h_j) \phi_j$ in
(\ref{eq:approxLLx}), we have
\begin{equation} \label{eq:approxLLx2}
\begin{aligned}
&- \sum_i  \int_\Omega   (m^{h,k}_i \times v^{h,k} ) \cdot  u^h_i \phi_i  =
 \eta \sum_l    \int_\Omega  \partial_{x_l}(m^{h,k}+\theta k \hat v^{h,k})   \cdot \partial_{x_l} u^h \\
&\hspace{2cm}-\eta\sum_{l,i}   \int_\Omega (\partial_{x_l} (m^{h,k}+\theta k \hat v^{h,k})   \cdot m^{h,k}_i)  \; ( m_i^{h,k} \cdot u^h_i)   \partial_{x_l} \phi_i \\
&\hspace{2cm}+ \alpha\eta \sum_{l,i}  \int_\Omega  (m_i^{h,k} \times \partial_{x_l} (m^{h,k}+\theta k \hat v^{h,k}) )\cdot  \partial_{x_l} \phi_i  u^h_i \\
&\hspace{2cm}- \sum_i  \int_\Omega ( \bar h (m^{h,k}+\theta k \hat v^{h,k})  ) \cdot   \phi_i  u^h_i \\
&\hspace{2cm}+ \sum_i   \int_\Omega (\bar h (m^{h,k}+\theta k \hat v^{h,k})   \cdot m^{h,k}_i)  \; ( m_i^{h,k} \cdot u^h_i)    \phi_i  \\
&\hspace{2cm}- \alpha \sum_i   \int_\Omega   (m_i^{h,k} \times \bar h (m^{h,k}+\theta k \hat v^{h,k}) ) \cdot u^h_i  \phi_i . 
\end{aligned}
\end{equation}
 Equations (\ref{eq:approxLLx}) and
  (\ref{eq:approxLLx2}) have terms that contain the product of weakly
  convergent sequences, namely the third term of the right hand side of 
  (\ref{eq:approxLLx}), and the second term of the right hand side of 
  (\ref{eq:approxLLx2}),  $\alpha\eta\sum_{l,i}
   \int_\Omega ( \partial_{x_l} (m^{h,k}+\theta k \hat
  v^{h,k}) \cdot m^{h,k}_i)(m^{h,k}_i \cdot w_i^h )\partial_{x_l} \phi_i
  .$
 By adding $\alpha$ times equation (\ref{eq:approxLLx2})
 to equation (\ref{eq:approxLLx}),    we eliminate the terms that contain the product of weakly
  convergent sequences :
\begin{equation} \label{eq:approxLLx3}
\begin{aligned}
&\int_\Omega \big[  v^{h,k}  \cdot w^h  - \alpha  \sum_i   (m^{h,k}_i \times v^{h,k} ) \cdot  w^h_i \phi_i   \big] \\ 
&\hspace{.1cm}=(1+\alpha^2) \bigg[ \eta  \sum_{l,i} \int_\Omega \big[  (m_i^{h,k} \times \partial_{x_l} m^{h,k})\cdot  \partial_{x_l} \phi_i  w^h_i  
+  \theta k   (m_i^{h,k} \times \partial_{x_l} \hat v^{h,k} )\cdot  \partial_{x_l} \phi_i  w^h_i \big] \\
&\hspace{.1cm}-  \sum_i \int_\Omega  \big[ (m_i^{h,k} \times \bar h (m^{h,k}) ) \cdot w^h_i  \phi_i 
+  \theta k   (m_i^{h,k} \times \bar h (\hat v^{h,k}) ) \cdot w^h_i  \phi_i\big]  \bigg]. 
\end{aligned}
\end{equation}
This is a similar procedure to subtracting $\alpha$ times the following equation
\begin{equation}\label{eq:LLmx}
m \times   \partial_t m = - m \times( m \times h) +\alpha m \times h
\end{equation}
from the Landau-Lifshitz equation (\ref{eq:LL}) to get the Landau-Lifshitz-Gilbert equation (\ref{eq:LLG}). 
Here, equation (\ref{eq:LLmx})  is obtained by taking $m\times$ the  Landau-Lifshitz equation  (\ref{eq:LL}).

\subsection{Energy inequality}
\label{subsec:energy}
In this section, we derive the energy inequalities we will need to prove Theorem \ref{thm:main}, namely
  (\ref{eq:energyestimate}) for $0 \leq \theta <\frac{1}{2}$ and
  (\ref{eq:energyestimate2}) for $ \frac{1}{2} \leq \theta \leq 1$.
We will use Theorem 1 from \cite{alouges2012convergent},
which states that the exchange energy is decreased after
  renormalization. This result goes back to
  \cite{alouges2004energetics,bartels2005stability} :
\begin{theorem} \label{thm:energydecrease}
For the $P^1$ approximation in $\Omega \subset \mbb{R}^2 $, 
if
\begin{equation} \label{eq:hy1}
\begin{aligned}
 \int_\Omega \nabla\phi_i \cdot \nabla\phi_j  \leq 0, \quad \text{ for } i \neq j,
  \end{aligned}
\end{equation}
then for all $w = \sum_{i=1}^N w_i \phi_i \in F^h$ such that $|w_i|
\geq 1$ for $i=1, \dots, N$, we have
\begin{equation} \label{eq:energydecrease}
\begin{aligned}
 \int_\Omega  \left| \nabla I_h (\frac{w}{|w|})   \right|^2 \leq \int_\Omega  \left| \nabla w \right|^2 .
 \end{aligned}
 \end{equation}
\end{theorem}
In 3D, we have (\ref{eq:energydecrease}), if  an additional condition that all dihedral
angles of the tetrahedra of the mesh are smaller than $\frac{\pi}{2}$ is satisfied,
along with (\ref{eq:hy1}).
 Also, we will use inequality (14) of \cite{alouges2012convergent}, 
\begin{equation}\label{eq:hbar}
\begin{aligned}
 \left\Vert\bar h( m)  \right\Vert_{L^2} \leq  C_5 \left\Vert m \right\Vert_{L^2} + C_5
  \end{aligned}
 \end{equation}
 and equation (25) from \cite{alouges2014convergent},
 \begin{equation}\label{eq:hbar2}
\begin{aligned}
 \left\Vert h_s( m)  \right\Vert_{L^2} \leq  C_5 \left\Vert m \right\Vert_{L^2} 
  \end{aligned}
 \end{equation}
where $C_5$ are positive constants, depending only on $\Omega$.
 Furthermore, we will use an inequality (20) of
\cite{alouges2012convergent} in the proof, which states there exists
$C_6>0$ such that for all $1 \leq p < \infty$ and all $\phi_h \in
F^h$, we have
\begin{equation} \label{eq:fix0}
  \frac{1}{C_6} \left\Vert \phi_h \right\Vert_{L^p}^p \leq h^d
  \sum_{i=1}^N |\phi_h(x_i)|^p \leq C_6 \left\Vert\phi_h\right\Vert_{L^p}^p.
\end{equation}
Moreover, we will assume that there exists $C_7>0$ such that
\begin{equation} \label{eq:fix1}
\int_\Omega |\nabla v^h|^2 \leq \frac{C_7}{h^2} \int_\Omega | v^h|^2 
\end{equation}
for all $v^h \in F^h$.

Taking $w^h=\sum_{j=1}^N (m_j^{h,k} \times u^h_j) \phi_j$ in 
(\ref{eq:approxLLx3}), and setting $u^h=\hat v^{h,k}$, we have
\begin{equation} \label{eq:fix2}
\begin{aligned}
 &- \alpha  \sum_i \int_\Omega   v^{h,k}_i \cdot  \hat v_i \phi_i  
= 
(1+\alpha^2) \bigg[ \eta  \sum_{l,i} \int_\Omega   \big[(\partial_{x_l} m^{h,k}\cdot  \partial_{x_l} \phi_i   \hat v_i)  +  \theta k   (\partial_{x_l} \hat v^{h,k}\cdot  \partial_{x_l} \phi_i    \hat v_i) \big] \\
&\hspace{5cm} 
-  \sum_i \int_\Omega  \big[ ( \bar h (m^{h,k})  \cdot   \hat v_i)  \phi_i 
+  \theta k   ( \bar h (\hat v^{h,k})  \cdot   \hat v_i)  \phi_i\big] \bigg]\\
\end{aligned}
\end{equation}
where we have used the fact $m_i^{h,k} \cdot \hat v_i^{h,k} = 0$ for $i=1, \dots, N$.
This equation can be written as
\begin{equation} \label{eq:innerproduct}
\begin{aligned}
(\nabla m, \nabla \hat v) = - \theta k \left\Vert \nabla \hat
  v\right\Vert_{L^2}^2 - \frac{\alpha}{1+\alpha^2}
  \frac{1}{\eta}\sum_i \frac{| (\mathbf{M}v)_j|^2}{b_j} +
  \frac{1}{\eta} ( \bar h(m), \hat v) + \frac{\theta k}{\eta}( \bar
  h(\hat v), \hat v).
\end{aligned}
\end{equation}

We now derive an energy estimate. We have
\begin{equation} \label{eq:energyinequality}
\begin{aligned}
  \frac{1}{2} & \left\Vert \nabla  m^{j+1} \right\Vert_{L^2}^2 \leq \frac{1}{2}
  \left\Vert \nabla m^j + k \nabla \hat{v}^j\right\Vert_{L^2}^2 = \frac{1}{2}
  \left\Vert \nabla m^j \right\Vert_{L^2}^2  + k (\nabla m^j, \nabla \hat{v}^j)
  + \frac{1}{2} k^2 \left\Vert  \nabla \hat{v}^j \right\Vert_{L^2}^2 \\
  & \leq \frac{1}{2} \left\Vert \nabla m^j \right\Vert_{L^2}^2 -
  k (\frac{\alpha}{1+\alpha^2} ) \frac{1}{\eta} \sum_i
  \frac{| (\mathbf{M}v)_i^j|^2}{b^2_i} b_i+ \frac{1}{2} k^2
  \left\Vert  \nabla \hat{v}^j \right\Vert_{L^2}^2  -
  \theta k^2 \left\Vert \nabla \hat v^j\right\Vert_{L^2}^2   \\
  & \quad \quad+  \frac{k}{\eta}( \bar h(m^j), \hat v^j) +  \theta
  \frac{k^2}{\eta}( \bar h(\hat v^j), \hat v^j) \\
  & \leq \frac{1}{2} \left\Vert \nabla m^j \right\Vert_{L^2}^2 -
  k (\frac{\alpha}{1+\alpha^2} )\frac{1}{\eta} \frac{C_1}{C_6}\left\Vert \hat v^j\right\Vert_{L^2}^2 -
  (\theta-\frac{1}{2}) k^2 \left\Vert \nabla \hat v^j\right\Vert_{L^2}^2 +
  \frac{k}{\eta}( \bar h(m^j), \hat v^j) \\
  & \quad \quad +  \theta \frac{k^2}{\eta} ( \bar h_e, \hat v^j) 
\end{aligned}
\end{equation}
where the first inequality is obtained by Theorem \ref{thm:energydecrease},  
the second inequality by equation (\ref{eq:innerproduct}),
and the last inequality by the fact $ (  h_s(\hat v^j), \hat v^j)  <0$.  
We have the estimate for the last two terms of the above inequality :
\begin{equation} \label{eq:young}
\begin{aligned}
  \left|( \bar h(m^j)+ \theta k \bar h_e, \hat v^j) \right| \leq \left\Vert
  \bar h(m^j)+ \theta k \bar h_e\right\Vert_{L^2} \left\Vert \hat v^j\right\Vert_{L^2} \leq
  C_8  \left\Vert \hat v^j\right\Vert_{L^2} \leq
  \epsilon \left\Vert \hat v^j\right\Vert^2_{L^2} +\frac{1}{4\epsilon} C_8^2
\end{aligned}
\end{equation}
for some $C_8 >0$, where the second inequality is obtained by equation (\ref{eq:hbar}) and the last inequality by Young's
inequality with $\epsilon =\frac{1}{2} \frac{\alpha}{1+\alpha^2}
\frac{C_1}{C_6}$.  Summing the inequality (\ref{eq:energyinequality})
from $j=0, \dots, J-1$ and using (\ref{eq:young}), we get
\begin{equation}
\begin{aligned}
  & \frac{1}{2} \left\Vert \nabla m^J \right\Vert_{L^2}^2 +
  k(\frac{1}{2\eta} (\frac{\alpha}{1+\alpha^2} ) \frac{C_1}{C_6} -C_7(\frac{1}{2}-\theta) \frac{k}{h^2} )
  \sum_{j=0}^{J-1}\left\Vert \hat v^j\right\Vert_{L^2}^2 \leq
  \frac{1}{2} \left\Vert \nabla m^0 \right\Vert_{L^2}^2+ C_9 T
  \end{aligned}
\end{equation}
with $\frac{k}{h^2} \leq C_0 < \frac{1}{2} \frac{\alpha }{1+\alpha^2}\frac{C_1}{C_6} \frac{1}{C_7 \eta}$,  for $0 \leq \theta <\frac{1}{2} $,
and 
\begin{equation}
\begin{aligned}
  &\frac{1}{2} \left\Vert \nabla m^J \right\Vert_{L^2}^2 
  +k(\frac{1}{2\eta} \frac{\alpha}{1+\alpha^2} ) \frac{C_1}{C_6}
  \sum_{j=0}^{J-1}\left\Vert \hat v^j\right\Vert_{L^2}^2+
  (\theta-\frac{1}{2}) k^2  \sum_{j=0}^{J-1}\left\Vert \nabla \hat v^j\right\Vert_{L^2}^2 \leq
  \frac{1}{2} \left\Vert \nabla m^0 \right\Vert_{L^2}^2 + C_9 T
   \end{aligned}
\end{equation}
for $\frac{1}{2} \leq \theta \leq 1$, and for some $C_9>0$.

In summary, we have the energy inequalities 
\begin{equation}\label{eq:energyestimate}
\begin{aligned}
   \frac{1}{2} \int_\Omega &| \nabla m^{h,k}(\cdot, T) |^2  +
  (\frac{1}{2\eta} (\frac{\alpha}{1+\alpha^2} )
  \frac{C_1}{C_6} -C_7C_0)  
  \int_{\Omega_T} | \hat v^{h,k}|^2  \\
 &  \leq
  \frac{1}{2} \int_\Omega |\nabla m^{h,k}(\cdot, 0) |^2+ C_9 T
  \end{aligned}
\end{equation}
with $ C_0 < \frac{1}{2} \frac{\alpha }{1+\alpha^2}\frac{C_1}{C_6} \frac{1}{C_7 \eta}$,  for $0 \leq \theta <\frac{1}{2} $ and 
\begin{equation}\label{eq:energyestimate2}
\begin{aligned}
  \frac{1}{2} \int_\Omega &\left| \nabla m^{h,k}(\cdot, T) \right|^2 \;  +
  (\frac{1}{2 \eta} \frac{\alpha}{1+\alpha^2} ) \frac{C_1}{C_6}
  \int_{\Omega_T} \left| \hat v^{h,k} \right|^2\;  \\
&  + (\theta-\frac{1}{2}) k \int_{\Omega_T} \left| \nabla \hat v^{h,k} \right|^2\;  
 \leq    \frac{1}{2} \int_\Omega \left| \nabla m^{h,k}(\cdot, 0) \right|^2+ C_9T.
   \end{aligned}
\end{equation}
for $\frac{1}{2} \leq \theta \leq 1$.
\subsection{Weak convergence of $m^{h,k}$, $\bar m^{h,k}$ and $\hat v^{h,k}$}
\label{subsection:weak}
In this section, we show the weak convergence of $\bar
m^{h,k}$ and $\hat v^{h,k}$ and strong convergence of  $m^{h,k}$ in some sense,
based on the energy estimates (\ref{eq:energyestimate}) and (\ref{eq:energyestimate2}). 
We follow similar arguments from section 6 of \cite{alouges2014convergent}.

Since we have
\begin{equation}\label{eq:k1}
 \left| \frac{m_i^{j+1}-m_i^j}{k}\right| \leq \left|\hat v_i^j\right|
 \end{equation}
for $ i=1, \dots, N$ and $j=0, \dots, J-1$, we have
\begin{equation}
  \left\Vert \partial_t \bar m^{h,k}\right\Vert_{L^2(\Omega)} =
  \left\Vert\frac{m^{j+1}-m^j}{k}\right\Vert_{L^2(\Omega)}
  \leq C_6 \left\Vert \hat{v}^{h,k} \right\Vert_{L^2(\Omega)}.
\end{equation}
Thus, we have
\begin{equation}\label{eq:timeder}
  \left\Vert \partial_t \bar m^{h,k}\right\Vert_{L^2(\Omega_T)} =
  \left\Vert\frac{m^{j+1}-m^j}{k}\right\Vert_{L^2(\Omega_T)} \leq
  C_6 \left\Vert \hat{v}^{h,k} \right\Vert_{L^2(\Omega_T)}
\end{equation}
which is bounded by the energy inequalities, (\ref{eq:energyestimate}) for $0 \leq \theta
<\frac{1}{2}$ and (\ref{eq:energyestimate2}) for $\frac{1}{2} \leq
\theta \leq 1$.  Hence, $\bar
m^{h,k}$ is bounded in $H^1(\Omega _T)$ and $\hat v^{h,k}$
is bounded in $L^2(\Omega _T)$ by (\ref{eq:timeder}) and by the energy inequalities,
(\ref{eq:energyestimate}) for $0 \leq \theta <\frac{1}{2}$ and
(\ref{eq:energyestimate2}) for $\frac{1}{2} \leq \theta \leq 1$. 
Thus, by passing to subsequences, there exist $m \in
H^1(\Omega _T)$ and $\hat v \in L^2(\Omega _T)$
such that
\begin{equation} \label{eq:weak}
\begin{aligned}
&\bar m^{h,k} \to m \text{ weakly in } H^1(\Omega _T),\\
&\bar m^{h,k} \to m \text{ strongly in } L^2(\Omega _T),\\
&\hat v^{h,k} \to \hat v \text{ weakly in } L^2(\Omega _T).
\end{aligned}
\end{equation}
Moreover, we have
\begin{equation}\label{eq:k2}
\begin{aligned}
  \left\vert m_i^{j+1}-m_i^j - k \hat v^j_i\right| =\left| \frac{m_i^j+k \hat v_i^j}{|m_i^j+k \hat v_i^j |}
  -m_i^j - k \hat v^j_i \right|  = \left| 1-|m_i^j+k \hat v_i^j|  \right| \leq
  \frac{1}{2}k^2 \left|\hat v_i^j \right|^2,
\end{aligned}
\end{equation}
since $|m_i^j+k \hat v_i^j|=\sqrt{1+k^2 |\hat v_i^j|^2} \leq
1+\frac{1}{2}k^2 |\hat v_i^j|^2$, for $ i=1, \dots, N$ and $j=0, \dots, J-1$.  Thus,
\begin{equation}\label{eq:mtv}
  \left\Vert \partial_t \bar m^{h,k}-\hat v^{h,k} \right\Vert_{L^1(\Omega _T)} \leq
  \frac{1}{2} k C_2 C_6 \left\Vert \hat v^{h,k}\right\Vert_{L^2(\Omega _T)}^2 
\end{equation}
which converges to $0$ as $h,k \to 0$, so 
\begin{equation} \label{eq:weakmt}
\partial_t m= \hat v.
\end{equation}

Furthermore,  since
\begin{equation}\label{eq:mhkweak}
  \left\Vert m^{h,k}-\bar m^{h,k}\right\Vert_{L^2(\Omega_T)} =
  \left\Vert (t-jk)\frac{m^{j+1}-m^j}{k}\right\Vert_{L^2(\Omega_T)} \leq
  k \left\Vert \partial_t \bar m^{h,k}\right\Vert_{L^2(\Omega_T)} 
\end{equation}
and the right hand side goes to $0$ as $h,k \to 0$, we have
\begin{equation}\label{eq:mhkstrongl2}
m^{h,k} \to m \text{ strongly in } L^2(\Omega _T).
\end{equation}

In summary, we have shown that there exist a subsequence of $ \{\bar m^{h,k}  \}$ that converges weakly in $H^1(\Omega \times(0,T))$,
a subsequence of $ \{\bar v^{h,k}  \}$ that converges weakly in $L^2(\Omega \times(0,T))$, 
 and a subsequence of $\{ m^{h,k} \}$ converges strongly in $L^2(\Omega_T)$ 
based on the energy estimates (\ref{eq:energyestimate}) and (\ref{eq:energyestimate2}). 
However, in our numerical tests in section \ref{sec:numerical}, 
it was not necessary to take subsequences and the method was in fact second order in space and first order in time.
Thus, there is still a gap in what we are able to prove and the practical performance of the algorithm in cases
where the weak solution is unique and sufficiently smooth.
\subsection{The proof that the limit $m$ actually satisfies Landau-Lifshitz-Gilbert equation}

\label{subsection:detail}
In this section, we show that each term of 
equation (\ref{eq:approxLLx3}) converges to the appropriate limit, so that the limit $m$ of the sequences $\{\bar m^{h,k}\}$  and $\{m^{h,k}\}$ satisfies the weak form of the Landau-Lifshitz-Gilbert equation  (\ref{eq:defLLG}) in Definition \ref{def:weak}.
\begin{lemma}\label{lemma:vmt}
Let the sequences $\{m^{h,k}\}$, $\{\bar m^{h,k}\}$, $\{\hat v^{h,k}\}$, and $\{v^{h,k}\}$ be defined by Definition \ref{def:main}. Also, let $m \in H^1(\Omega _T)$ be the limit as in (\ref{eq:weak}) and (\ref{eq:mhkstrongl2}).
 Moreover, let's assume $ w\in (C^\infty(\Omega _T)^3 \cap (H^1(\Omega _T))^3$, and $w^h = I_h (w) \in F_h$ as in equation (\ref{eq:interpolation}). 
 Then we have
\begin{equation}\label{eq:lhsold}
\begin{aligned}
  \lim_{h,k \to 0}\int_{\Omega_T}  v^{h,k} \cdot w^h   =
  \lim_{h,k \to 0}\int_0^T  \sum_{j=1}^N \hat v_j^{h,k} \cdot w_j^h \int_\Omega \phi_j  
  = \int_{\Omega_T} \partial_t m \cdot w    .
\end{aligned}
\end{equation}
\end{lemma}
\begin{proof}
 The difference between the last two terms is bounded by 
\begin{equation}\label{eq:v1}
\begin{aligned}
\left| \int_{\Omega_T} I_h(\hat v^{h,k} \cdot  w^h) - \hat v^{h,k} \cdot  w^h \right| 
+ \left| \int_{\Omega_T} \hat v^{h,k} \cdot w^h- \partial_t m  \cdot w \right|.
\end{aligned}
\end{equation}
The first term of  (\ref{eq:v1}) has the following
estimate. For each element $L$, we have $\hat v^{h,k}(\cdot, t) \cdot w^h(\cdot,t) \in
C^\infty(L)$ and
\begin{equation}
\begin{aligned}
  &\left\Vert I_h(\hat v^{h,k} \cdot w^h)- \hat v^{h,k} \cdot  w^h\right\Vert_{L^2(L)}^2 \leq C_{10}h^4
  \left\Vert \Delta (\hat v^{h,k} \cdot  w^h) \right\Vert_{L^2(L)}^2  \\
  &\leq C_{10}h^4 ( \left\Vert \Delta \hat  v^{h,k} \cdot w^h \right\Vert^2_{L^2(L)}+
  \left\Vert \nabla  \hat v^{h,k}  \cdot \nabla w^h \right\Vert^2_{L^2(L)}+\left\Vert \hat v^{h,k}
  \cdot \Delta w^h \right\Vert_{L^2(L)}^2)\\
  &\leq C_{10}h^4 ((\left\Vert \nabla \hat  v^{h,k}\cdot \nabla w^h \right\Vert^2_{L^2(L)}  )
\end{aligned}
\end{equation}
for some $C_{10}>0$, where the first inequality is obtained by the Bramble-Hilbert lemma \cite{braess2007finite}, 
and in the last inequality we have used $\Delta \hat v^{h,k} = 0 $ and $\Delta w^h =0$ in $L$, since $\hat v^{h,k}$  and $w^h$ are the sum of piecewise linear functions. We have the estimate 
\begin{equation} \label{eq:vh}
\begin{aligned}
  \left\Vert \nabla\hat  v^{h,k}\right\Vert_{L^2(\Omega)}^2  \leq  \sum_L  \int_L |\sum_i
  ( \hat v^{h,k})_i \nabla \phi_i |^2   \leq \frac{C_{11}}{h^2}
  \sum_L   |\sum_{i\in I_L }(\hat v^{h,k})_i  |^2 |L|\\
 \leq C_{12} h^{d-2}   \sum_{i=1}^{N} |(\hat v^{h,k})_i  |^2 
 \leq \frac{C_{13}}{h^2}\left\Vert \hat v^{h,k}\right\Vert_{L^2(\Omega)}^2.
 \end{aligned}
\end{equation}
for some constants $C_{11}, C_{12}, C_{13}>0$ and $I_L$ is the
index of nodes of $L$, where the second inequality is obtained by (\ref{eq:fix1}), and the last inequality by (\ref{eq:fix0}). 
 Hence,
\begin{equation}
\begin{aligned}
\left\Vert I_h(\hat v^{h,k}\cdot w^h)- \hat v^{h,k}\cdot  w^h\right\Vert_{L^2(\Omega _T)} ^2
 &\leq C_{10}h^4 \left\Vert \nabla \hat  v^{h,k}\cdot  \nabla w^h \right\Vert^2_{L^2(\Omega _T)} 
  \leq C_{14} h^2\left\Vert  \hat  v^{h,k} \right\Vert_{L^2(\Omega _T)} 
\end{aligned}
\end{equation}
for some constant $C_{14}>0$. Therefore, the first term of (\ref{eq:v1}) goes to $0$ as $h,k \to 0$. Moreover, the second term of (\ref{eq:v1}) 
goes to $0$ by the weak convergence of $\hat v^{h,k}$ to $\partial_t m $ which are equations (\ref{eq:weak}) and (\ref{eq:weakmt}) .
\end{proof}

\begin{lemma} Under the same assumptions of Lemma \ref{lemma:vmt}, we have
\begin{equation}\label{eq:gilbert}
\begin{aligned}
\lim_{h,k \to 0} \int_{\Omega_T} \sum_i  (m^{h,k}_i \times v^{h,k} ) \cdot  w_i^h \phi_i  
&=\lim_{h,k \to 0} \int_{\Omega_T} \sum_i  (m^{h,k}_i \times \hat v_i^{h,k} ) \cdot  w_i^h \phi_i   \\
&=\int_{\Omega_T}  (m \times \partial_t m  ) \cdot  w 
 \end{aligned}
\end{equation}
\end{lemma}
\begin{proof}
 The difference between the last two terms is bounded by 
\begin{equation}\label{eq:vx}
\begin{aligned}
 & \left| \int_{\Omega_T} I_h( (m^{h,k})^a (\hat v^{h,k})^b (w^h)^c) - (m^{h,k})^a (\hat v^{h,k})^b (w^h)^c \right| \\
&+ \left| \int_{\Omega_T} (m^{h,k})^a (\hat v^{h,k})^b (w^h)^c-m^a (\partial_t m)^b w^c \right|
\end{aligned}
\end{equation}
for some $a,b,c \in \{1,2,3\}$.  The first term of (\ref{eq:vx}), has the
following estimate. For each element $L$, we have $(m^{h,k})^a (\hat
v^{h,k})^b (w^h)^c \in C^\infty(L)$ and
\begin{equation}
\begin{aligned}
&\left\Vert I_h((m^{h,k})^a (\hat v^{h,k})^b (w^h)^c) - (m^{h,k})^a (\hat v^{h,k})^b (w^h)^c \right\Vert_{L^1(L)} \\
&\leq C_{15} h^2 ( \left\Vert \Delta ((m^{h,k})^a (\hat v^{h,k})^b (w^h)^c) \right\Vert_{L^1(L)})\\
  &\leq C_{15} h^2 ( \left\Vert \nabla (m^{h,k})^a \nabla(\hat v^{h,k})^b (w^h)^c \right\Vert_{L^1(L)} +
  \left\Vert \nabla (m^{h,k})^a (\hat v^{h,k})^b \nabla (w^h)^c \right\Vert_{L^1(L)} \\
  &\hspace{1.5cm}  + \left\Vert  (m^{h,k})^a \nabla(\hat v^{h,k})^b \nabla (w^h)^c \right\Vert_{L^1(L)}   )
\end{aligned} 
\end{equation}
for some constant $C_{15}>0$, where the first inequality is obtained by Bramble-Hilbert lemma, and in the last inequality 
we have used $\Delta \hat m^{h,k} = 0 $, $\Delta \hat v^{h,k} = 0 $ and $\Delta w^h =0$ in $L$, since $m^{h,k}$ $\hat v^{h,k}$ 
 and $w^h$ are the sum of piecewise linear functions. Hence, we have the
estimate
\begin{equation}
\begin{aligned}
&\left\Vert I_h((m^{h,k})^a (\hat v^{h,k})^b (w^h)^c) - (m^{h,k})^a (\hat v^{h,k})^b (w^h)^c \right\Vert_{L^1(\Omega_T)} \\
  &\leq C_{16} h\left\Vert (\hat v^{h,k})^b \right\Vert_{L^2(\Omega_T)} ( \left\Vert \nabla (m^{h,k})^a\right\Vert_{L^2(\Omega_T)}  
 +h\left\Vert \nabla (m^{h,k})^a \right\Vert_{L^2(\Omega_T)}  \\
& + \left\Vert  (m^{h,k})^a \right\Vert_{L^2(\Omega_T)}  ) .
\end{aligned} 
\end{equation}
for some constant $C_{16}>0$, where we have used H\"{o}lder's inequality for all the terms and used (\ref{eq:vh})  for the first and the third terms.  Therefore, the first term of (\ref{eq:vx}) goes to $0$ as $h,k \to 0$. 
Moreover, the second term of (\ref{eq:vx}) 
goes to $0$ by the weak convergence of $(\hat v^{h,k})^b$ to $(\partial_t m)^b $ established in (\ref{eq:weak}) and (\ref{eq:weakmt}), and strong convergence of $(m^{h,k})^a$ to
$m^a$.
\end{proof}

\begin{lemma} \label{lemma:gyro:old}Under the same assumptions of Lemma \ref{lemma:vmt}, we have
\begin{equation}
\begin{aligned}
  \lim_{h,k \to 0}\sum_{l,i} \int_{\Omega_T} (m^{h,k}_i \times  \partial_{x_l}m^{h,k}) \cdot
 w_i^h \;  \partial_{x_l} \phi_i  \; = 
   \sum_l \int_{\Omega_T} (m \times \partial_{x_l} m ) \cdot  \partial_{x_l} w \;.
\end{aligned}
\end{equation}
\end{lemma}
\begin{proof} 
 The difference between the last two terms is bounded by 
\begin{equation}
\begin{aligned}\label{eq:gyro2}
    \left\vert\int_{\Omega_T} ( \partial_{x_l} {m^{h,k}})^{b}    \partial_{x_l} I_h((m^{h,k})^{c}(w^h) ^{a})\;  -
  \int_{\Omega_T}  ( \partial_{x_l} {m^{h,k}})^{b}   ((m^{h,k})^{c} (\partial_{x_l} w^h)^{a})\;  \right\vert \\
  +    \left\vert\int_{\Omega_T}  (m^{h,k})^{c}( \partial_{x_l} {m^{h,k}})^{b}   (\partial_{x_l} w^h)^{a} \; -
  \int_{\Omega_T}   m^{c}  (\partial_{x_l} m)^{b}  ( \partial_{x_l} w)^{a})\; \right\vert ,
\end{aligned}
\end{equation}
for some $a,b,c \in \{ 1,2,3 \}$. The first term is bounded by
\begin{equation}\label{eq:gyro3}
  \left\Vert ( \partial_{x_l} m^{h,k})^{b}\right\Vert_{L^2(\Omega _T)}
  \left\Vert \partial_{x_l} I_h((m^{h,k})^{c}(w^h)^{a}) -
  \partial_{x_l} ((m^{h,k})^{c}(w^h)^{a}) \right\Vert_{L^2(\Omega _T)}.
\end{equation}
For each element $L$, we have $m^{h,k}(\cdot, t) w (\cdot, t) \in C^{\infty}(L)$, and we have the estimate,
\begin{equation}
\begin{aligned}\label{eq:BH1}
  \left\Vert \partial_{x_l} I_h({(m^{h,k}})^{c}(w^h)^a) -\partial_{x_l}
  (({m^{h,k}})^{c}(w^h)^a) \right\Vert_{L^2(L)}^2 
\leq C_{17} h^2 |({m^{h,k}})^{c}(w^h)^a|_{H^2(L)}^2
\end{aligned}
\end{equation}
for some constant $C_{17}>0$, by the Bramble-Hilbert lemma. Moreover, we have the estimate, 
\begin{equation}
\begin{aligned}\label{eq:BH2}
|({m^{h,k}})^{c}(w^h)^a|_{H^2(L)}^2 = \int_L |\Delta ((m^{h,k})^{c}(w^h)^a)|^2 
  \leq C_{18}  \int_L |  \nabla (m^{h,k})^{c} |^2| \nabla (w^h)^a|^2\\
\leq C_{19} \left\Vert(m^{h,k})^c\right\Vert_{H^1(L)}^2
\end{aligned}
\end{equation}
for some constants $C_{18}, C_{19}>0$, since $\Delta m^{h,k} = 0 $ and $\Delta w^h =0$ in $L$, since $m^{h,k}$  and $w^h$ are the sum of piecewise linear functions. We get the estimate
\begin{equation}
\begin{aligned}
  \left\Vert \partial_{x_l} I_h({(m^{h,k}})^{c}(w^h)^a) -\partial_{x_l}
  (({m^{h,k}})^{c}(w^h)^a) \right\Vert_{L^2(\Omega_T)}^2 
   \leq C_{17}C_{19} h^2
  \left\Vert(m^{h,k})^c \right\Vert_{H^1(\Omega_T)}^2 .
\end{aligned}
\end{equation}
Therefore, we may conclude that the first term
of (\ref{eq:gyro2}) goes to 0 as $h,k \to 0$.
Moreover, the second term of  (\ref{eq:gyro2}) 
goes to $0$ by the weak convergence of $(\partial_{x_l} m^{h,k})^b$ to $(\partial_{x_l} m)^b $ and strong convergence of $(m^{h,k})^c$ to
$m^c$,  which gives (\ref{eq:weak}) and (\ref{eq:mhkstrongl2}).

\end{proof}

\begin{lemma} \label{lemma:gyrov}Under the same assumptions of Lemma \ref{lemma:vmt}, we have
\begin{equation} \label{eq:gyrov}
\begin{aligned}
  & \lim_{h,k \to 0} \left\vert k \sum_i \int_{\Omega_T} (m^{h,k}_i \times
  \partial_{x_l}\hat v^{h,k} )^a ( \partial_{x_l} w_i^h )^a  \right\vert =0.
\end{aligned}
\end{equation}
for $0 \leq \theta \leq 1$.
\end{lemma}
\begin{proof}
 An upper bound for the sequence above
is
\begin{equation} \label{eq:kv}
\begin{aligned}
  &  \sqrt{k} \left\Vert\sqrt{k} \;\partial_{x_l} (\hat v^{h,k} )^{c}\right\Vert_{L^2(\Omega _T)}
  \left\Vert \nabla (I_h(  m^{h,k})^{b}   (w^h)^a ) \right\Vert_{L^2(\Omega _T)}.
\end{aligned} 
\end{equation}
for some $a,b,c \in \{ 1,2,3 \}$.  The term $\left\Vert\sqrt{k} \;\partial_{x_l} (\hat v^{h,k} )^{c}\right\Vert_{L^2(\Omega _T)}$ in  (\ref{eq:kv}) is uniformly bounded, since 
$ \left\Vert \sqrt{k} \partial_{x_l} (\hat v^{h,k} )^{c}\right\Vert_{L^2(\Omega) } \leq  C_7 \frac{\sqrt{k}}{h}\left\Vert(\hat v^{h,k} )^{c}\right\Vert_{L^2(\Omega) } $
is uniformly bounded  by (\ref{eq:energyestimate}) for $0
\leq \theta <\frac{1}{2}$, which is obtained by (\ref{eq:fix1}), 
and $ \left\Vert \sqrt{k} \partial_{x_l} (\hat v^{h,k} )^{c}\right\Vert_{L^2 (\Omega)}$ is uniformly bounded by equation
(\ref{eq:energyestimate2}) for $\frac{1}{2} \leq \theta \leq 1$.
For each element $L$, we have $m^{h,k}(\cdot, t) w (\cdot, t) \in C^{\infty}(L)$, so 
\begin{equation} 
\begin{aligned}
  & \left\Vert \nabla I_h((  m^{h,k})^{b}   (w^h)^a )- \nabla ((  m^{h,k})^{b}   (w^h)^a ) \right\Vert_{L^2(L)}^2 
\leq C_{20} h^2 (\left\Vert  ( \nabla (  m^{h,k})^{b}   ) \right\Vert_{L^2(L)}^2),
\end{aligned} 
\end{equation}
for some constant $C_{20}>0$, by the Bramble-Hilbert lemma, and using $\Delta m^{h,k} = 0 $ and $\Delta w^h =0$ in $L$, since $m^{h,k}$  and $w^h$ are the sum of piecewise linear functions.
Thus, we have
\begin{equation} 
\begin{aligned}
& \left\Vert \nabla (I_h(  m^{h,k})^{b}   (w^h)^a ) \right\Vert_{L^2(\Omega_T)}^2 \\
&\hspace{3cm}\leq \left\Vert \nabla ( m^{h,k})^{b}    \right\Vert_{L^2(\Omega_T)}^2+ C_{20} h^2 (\left\Vert
  ( \nabla (  m^{h,k})^{b}   ) \right\Vert_{L^2(\Omega_T)}^2 ),
\end{aligned} 
\end{equation}
which is uniformly bounded.
Hence, (\ref{eq:kv}) goes to $0$ as $h,k \to 0$.
\end{proof}

\begin{lemma} Under the same assumptions of Lemma \ref{lemma:vmt}, we have
\begin{equation}
\begin{aligned}
  \lim_{h,k \to 0}  \sum_i \int_{\Omega_T}
  (m_i^{h,k} \times \bar h (m^{h,k}) ) \cdot w_i^h  \phi_i  =
 \int_{\Omega_T}   (m \times \bar h (m) ) \cdot w .
\end{aligned}
 \end{equation}
\end{lemma}
\begin{proof}An upper bound for the difference between the sequence and the limit 
is given by
\begin{equation} \label{eq:gyroh}
\begin{aligned}
  \left\vert \int_{\Omega_T}    (\bar h (m^{h,k}) )^a I_h((m^{h,k})^b  (w^h)^c) -   \int_{\Omega_T}
 ( \bar h (m^{h,k}) )^a (m^{h,k})^b  (w^h)^c) \right\vert\\
  +\left\vert \int_{\Omega_T}    (\bar h (m^{h,k}) )^a (m^{h,k})^b  (w^h)^c     -  \int_{\Omega_T}
  (\bar h (m) )^a m^b  w^c)    \right\vert
\end{aligned}
 \end{equation}
 for some $a,b,c \in \{1,2,3 \}$.
 The first term of (\ref{eq:gyroh}) is bounded by
\begin{equation} 
\begin{aligned}
  \left\Vert \bar h (m^{h,k})^a \right\Vert_{L^2 (\Omega _T )} \left\Vert I_h((m^{h,k})^b  (w^h)^c) -
  (m^{h,k})^b  (w^h)^c\right\Vert_{L^2 (\Omega _T )}
\end{aligned}
 \end{equation}
For each element $L$, we have $m^{h,k}( \cdot, t) w( \cdot, t) \in C^{\infty} (L)$, and we get the estimate,
\begin{equation}
\begin{aligned}\label{eq:fix3}
 \left\Vert I_h((m^{h,k})^{b}w^{c}) - ((m^{h,k})^{b}w^{c}) \right\Vert_{L^2(L)}^2 
\leq C_{21} h^4 |(m^{h,k})^{b}w^{c}|_{H^2(L)}^2
\end{aligned}
\end{equation}
for some constant $C_{21}>0$, by the Bramble-Hilbert lemma. Moreover, 
\begin{equation}
\begin{aligned}\label{eq:fix4}
  |(m^{h,k})^{b}w^{c}|_{H^2(L)}^2 
    & \leq C_{21} \int_L |  \nabla (m^{h,k})^{b} |^2| \nabla (w^h)^c|^2+|(m^{h,k})^b|^2| \Delta (w^h)^c|^2\\
 &\leq C_{22} \left\Vert(m^{h,k})^b\right\Vert_{H^1(L)}^2
\end{aligned}
\end{equation}
for some constant $C_{22} >0$, and using the fact $\Delta m^{h,k} =\Delta w^h =0$ in $L$, since $m^{h,k}$  and $w^h$ are the sum of piecewise linear functions. We get the estimate
\begin{equation}
\begin{aligned}\label{eq:fix5}
  \left\Vert I_h((m^{h,k})^{b}w^{c}) - ((m^{h,k})^{b}w^{c}) \right\Vert_{L^2(\Omega _T)}^2 \leq
  C_{23} h^4 \left\Vert(m^{h,k})^b\right\Vert_{H^1(\Omega _T)}^2.
\end{aligned}
\end{equation}
for some constant $C_{23} >0$.  Thus, the first term of (\ref{eq:gyroh})
goes to $0$ as $h,k \to 0$, and the second term of (\ref{eq:gyroh}) converges to $0$ as $h,k \to 0$,
because of the strong convergence of $(\bar h( m^{h,k}))^a$ and
$(m^{h,k})^b$.
\end{proof}
\begin{lemma} Under the same assumptions of Lemma \ref{lemma:vmt}, we have
\begin{equation}
\begin{aligned}
  \lim_{h,k \to 0}  \left\vert k \sum_i \int_{\Omega_T}   (m_i^{h,k} \times \bar h
  (\hat v^{h,k}) ) \cdot w_i^h  \phi_i   \right\vert =0.
\end{aligned}
 \end{equation}
\end{lemma}
\begin{proof}
An upper bound for the sequence above
 is
\begin{equation}
\begin{aligned} \label{eq:khw}
  k \left\Vert \bar h (\hat v^{h,k}) )\right\Vert_{L^2(\Omega_T )}\left\Vert w^h \right\Vert_{L^2(\Omega_T )}.  
\end{aligned}
 \end{equation}
Since, 
$
  \left\Vert \bar h (\hat v^{h,k}) )\right\Vert_{L^2(\Omega_T )}  \leq ( C_5\left\Vert \hat v^{h,k}\right\Vert_{L^2(\Omega_T)} + C_5)
$ by  (\ref{eq:hbar}),  the term $\left\Vert \bar h (\hat v^{h,k}) )\right\Vert_{L^2(\Omega_T)}$ in (\ref{eq:khw}) is uniformly bounded. Therefore, (\ref{eq:khw}) goes to $0$ as $h,k \to 0$.

\end{proof}
%

\subsection{Energy of $m$}
\label{subsection:energy}
Recall the definition of the energy $\mathcal{E}(m)$  in  (\ref{eq:LLenergydef}). 
We  follow the same arguments in section 6 of \cite{alouges2014convergent}. 
We have an energy estimate  of $m^{h,k}$ as
\begin{equation}\label{energy:main1}
\begin{aligned}
  \mathcal{E}(m^{j+1}) - \mathcal{E}(m^{j}) \leq &	 - k (\frac{\alpha}{1+\alpha^2} )
  \frac{C_1}{C_6}|| \hat v^j||_{L^2}^2 - (\theta-\frac{1}{2}) k^2 \eta|| \nabla \hat v^j||_{L^2}^2+
  k( \bar h(m^j), \hat v^j)  \\
  & +  \theta k^2 ( \bar h_e, \hat v^j)- \frac{1}{2} \int_\Omega (\bar h (m^{j+1} ) + \bar h (m^j)  )\cdot (m^{j+1}-m^j) .
\end{aligned}
\end{equation}
by (\ref{eq:energyinequality}) from section \ref{subsec:energy}. 
For $0\leq \theta  <\frac{1}{2}$, the second term on the right has an upper bound
\begin{equation}
(\theta-\frac{1}{2}) k^2 \eta \left\Vert \nabla \hat v^j\right\Vert_{L^2(\Omega) }^2 \ \leq  k^2 \eta \left\Vert \nabla \hat v^j \right\Vert_{L^2(\Omega) }^2 \leq C_7 k \eta\frac{ k }{h^2} \left\Vert  \hat v^j \right\Vert_{L^2(\Omega) }^2 \leq C_7C_0 \eta k \left\Vert  \hat v^j \right\Vert_{L^2(\Omega) }^2
\end{equation}
and by choosing $ C_0 \leq \frac{1}{2} \frac{\alpha }{1+\alpha^2}\frac{C_1}{C_6} \frac{1}{C_7 \eta}$, this term and the first term on the right hand side of (\ref{energy:main1}) can be combined to be less than equal to 
\begin{equation}
 -\frac{k}{2}(\frac{\alpha}{1+\alpha^2} )  \frac{C_1}{C_6}|| \hat v^j||_{L^2(\Omega)}^2 
 \end{equation}
 The second term on the right of equation(\ref{energy:main1}) can be disregarded for $\frac{1}{2} \leq \theta \leq 1$.
We will derive the upper bound for the rest of the terms of right hand side of  (\ref{energy:main1}). 
\par The third and the last terms on the right can be combined to be written as
\begin{equation}
\begin{aligned}
& \left\vert  k( \bar h(m^j), \hat v^j) - \frac{1}{2} \int_\Omega (\bar h (m^{j+1})+\bar h (m^j)) \cdot (m^{j+1}-m^j)  \right \vert .
 \end{aligned}
\end{equation}
and has an upper bound
\begin{equation}\label{eq:up1}
\begin{aligned}
& \left\vert \int_\Omega \bar h(m^j)  \cdot (m^{j+1}-m^j- k \hat{v}^j ) \right\vert  + \left\vert \frac{1}{2} \int_\Omega (\bar h (m^{j+1})-\bar h (m^j)) \cdot (m^{j+1}-m^j)  \right\vert.
 \end{aligned}
\end{equation}
The first term of (\ref{eq:up1})  is bounded by 
\begin{equation}
\begin{aligned}
 C_{24} k^2  \left(\left\Vert  \hat v^j \right\Vert_{L^2(\Omega)} \left\Vert  \hat v^j \right\Vert_{L^4(\Omega)} \right)  \leq  C_{24} \frac{k^2}{2}  \left(\left\Vert  \hat v^j \right\Vert_{L^2(\Omega)}^2+ \left\Vert  \hat v^j \right\Vert_{L^4(\Omega)}^2\right) 
 \end{aligned}
\end{equation}
for some constant $C_{24}>0$,  by (\ref{eq:k2}), and (\ref{eq:fix0}).
The second term of (\ref{eq:up1})  is bounded by 
$C_{25} k^2  \left\Vert  \hat v^j \right\Vert^2_{L^2(\Omega)}$
 for some constant $C_{25}>0$, by (\ref{eq:k1}) and (\ref{eq:fix0}).

The fourth term on the right has the upper bound
$|  \theta k^2 ( \bar h_e, \hat v^j) | \leq C_{26} k^2   \left\Vert \hat v^j\right\Vert_{L^2(\Omega)}$
for some constant $C_{26} >0$.
Then (\ref{energy:main1}) has an upper bound
\begin{equation}
\begin{aligned}
  \mathcal{E}(m^{j+1}) - \mathcal{E}(m^{j})+ \frac{k}{2} \left(\frac{\alpha}{1+\alpha^2} \right)  \frac{C_1}{C_6}|| \hat v^j||_{L^2(\Omega)}^2    \leq &  C_{27} k^2  \left(\left\Vert  \hat v^j \right\Vert_{L^4(\Omega)}^2+\left\Vert  \hat v^j \right\Vert^2_{L^2(\Omega)} \right) \\
  \leq & C_{28} k^2  \left( \left\Vert \nabla \hat v^j \right\Vert_{L^2(\Omega)}^2 +\left\Vert  \hat v^j \right\Vert^2_{L^2(\Omega)} \right)
\end{aligned}
\end{equation}
for some constants $C_{27}, C_{28} >0$, by using Sobolev embedding theorem \cite{adams2003sobolev},
$\left\Vert  \hat v^j \right\Vert_{L^4(\Omega)} \leq C_{29} \left\Vert \nabla \hat v^j \right\Vert_{L^2(\Omega)}$
for some constant $C_{29}>0$. Summing from $j=0, \dots, J-1$, we get
\begin{equation}
\begin{aligned}
 & \mathcal{E}(m^{J}) - \mathcal{E}(m^{0})+  \frac{1}{2} \left(\frac{\alpha}{1+\alpha^2} \right)  \frac{C_1}{C_6}\int_{\Omega_T} |\hat v^{h,k}|^2    \\
  &\hspace{3cm} \leq  C_{28} \;k  \left(\left\Vert \nabla \hat v^{h,k} \right\Vert_{L^2(\Omega _T) }^2 +\left\Vert  \hat v^{h,k} \right\Vert^2_{L^2(\Omega _T)}  \right)
\end{aligned}
\end{equation}
Therefore, taking $h, k \to 0$, we get
the energy inequality (\ref{eq:energyiinequalitydef}).
 
\subsection{Magnitude of $m$}
\label{subsection:magnitude}
By the same argument in  \cite{alouges2012convergent}, we have $|m(x,t)|=1$ $a.e.$ for $(x,t)\in \Omega _T$ ( See equation (28) and (29) on page 1347 of  \cite{alouges2012convergent} ).
%
%

\section{Conclusion}
We have presented a mass-lumped finite element method for the Landau-Lifshitz equation.
We showed that the numerical solution of our method has a subsequence that converges weakly to a weak solution of the Landau-Lifshitz-Gilbert equation.
Numerical tests show that the method is second order accurate in space and first order accurate in time
when the underlying solution is smooth. A second-order in time variant was also presented and tested
numerically, but not analyzed rigorously in the present work.
 \vspace*{-.3cm}

\bibliography{refs}

\end{document}